\documentclass[11pt,a4paper]{amsart}
\usepackage[T1]{fontenc}
\usepackage[utf8]{inputenc}
\usepackage{lmodern}
\usepackage[margin=30mm]{geometry}
\usepackage{microtype}
\usepackage[all]{xy}
\usepackage{amsmath,amssymb,amsthm,mathtools,mathrsfs}
\usepackage{booktabs}
\usepackage{array}
\usepackage[shortlabels]{enumitem}
\usepackage{tikz-cd}
\usepackage{xcolor}
\usepackage[colorlinks=true,linkcolor=blue!55!black,
  citecolor=blue!55!black,urlcolor=blue!55!black]{hyperref}
\usepackage[nameinlink,noabbrev]{cleveref}

\newcommand{\R}{\mathbb R}
\newcommand{\er}{\mathbb R}
\newcommand{\N}{\mathbb N}
\newcommand{\F}{\mathcal F}

\newcommand{\g}{\boldsymbol{g}}
\newcommand{\f}{\boldsymbol{f}}

\newcommand{\e}{\boldsymbol{e}}
\newcommand{\uu}{\boldsymbol{u}}
\newcommand{\vv}{\boldsymbol{v}}
\newcommand{\bmu}{\boldsymbol{\mu}}

\newcommand{\Lipz}{\operatorname{Lip}_0}
\newcommand{\Lip}{\operatorname{Lip}}

\newcommand{\linspan}{\operatorname{span}}
\newcommand{\supp}{\operatorname{supp}}

\newcommand{\dist}{\operatorname{dist}}
\newcommand{\co}{\operatorname{co}}

\newcommand{\norm}[1]{\left\lVert #1\right\rVert}
\newcommand{\abs}[1]{\left\lvert #1\right\rvert}
\newcommand{\pair}[2]{\left\langle #1,#2\right\rangle}
\newcommand{\Int}{\operatorname{Int}}
\newcolumntype{P}[1]{>{\raggedright\arraybackslash}p{#1}}
\newcommand{\linearization}[1]{\widetilde{#1}}
\newcommand{\Cyl}{\operatorname{Cyl}}

\theoremstyle{plain}
\newtheorem{theorem}{Theorem}[section]
\newtheorem{proposition}[theorem]{Proposition}
\newtheorem{lemma}[theorem]{Lemma}
\newtheorem{corollary}[theorem]{Corollary}
\newtheorem{fact}[theorem]{Fact}
\theoremstyle{definition}

\theoremstyle{remark}
\newtheorem{remark}[theorem]{Remark}

\theoremstyle{plain}
\newtheorem{thmx}{Theorem}

\makeatletter
\newcommand{\nameditem}[2]{%
    \item[\textup{(#1)}]%
    \phantomsection
    \def\@currentlabel{\textup{(#1)}}%
    \label{#2}%
}
\makeatother

\begin{document}
\title[Lipschitz-free spaces and complementability]{Lipschitz-free spaces over separable Banach spaces and their complementability in the bidual}
\author[{M. C\' uth}]{Marek C\'uth}
\author[{O. Kalenda}]{Ond\v{r}ej F.K. Kalenda}

\email{cuth@karlin.mff.cuni.cz}
\email{kalenda@karlin.mff.cuni.cz}

\address{Department of Mathematical Analysis \\
Faculty of Mathematics and Physics\\ Charles University\\
Sokolovsk\'{a} 83, 186 \ 75\\Praha 8, Czech Republic}
\email{cuth@karlin.mff.cuni.cz}
\email{kalenda@karlin.mff.cuni.cz}


\begin{abstract} We prove that the Lipschitz-free space over a separable Banach space is complemented in its bidual if and only if the underlying Banach space is isomorphic to a Hilbert space. This answers negatively a long-standing question of Godefroy, Lancien, and Zizler concerning the case of $\ell_1$. We also characterize the existence of a projection of norm one and extend both characterizations to nonempty open convex subsets. Our approach yields several further results which we believe may be useful for understanding the structure of Lipschitz-free spaces over Banach spaces. These include an isometric representation of such spaces as quotients of Bochner $L^1$ spaces constructed using Gaussian measures.
\end{abstract}
\maketitle

\section{Introduction and main results}

Lipschitz-free spaces, also known as Arens--Eells spaces or transportation cost spaces, are Banach spaces canonically associated with metric spaces. Since the fundamental work of Godefroy and Kalton~\cite{GodefroyKalton}, they have become one of the most actively studied classes of Banach spaces, with strong connections to both linear and nonlinear Banach space theory. Recent advances illustrate, for example, their connections with metric geometry, with applications to geometric group theory~\cite[Theorem~E]{Gartland}, with operator theory through metric characterizations of
properties of Lipschitz linearizations~\cite{FJLPPQ25}, or with the extension of Lipschitz maps with values in Banach spaces~\cite[Section~5.2]{Naor24}. Ideas from Choquet theory have also inspired recent work on representing measures in Lipschitz-free spaces~\cite{S26}.

In this paper we study when Lipschitz-free spaces over Banach spaces are complemented in their biduals. A long-standing problem in this direction, formulated in the survey of Godefroy, Lancien, and Zizler~\cite[Problem~16]{GLZ14}, asks whether $\F(\ell_1)$ is complemented in its bidual. A positive answer would show that $\ell_1$ is determined by its Lipschitz structure: every Banach space Lipschitz isomorphic to $\ell_1$ would then be linearly isomorphic to $\ell_1$. The problem has attracted sustained attention and has been repeated in the monographs of Guirao, Montesinos, and Zizler~\cite[Problem~235]{GMZ16} and Baudier and Lancien~\cite[Problem~11]{BaudierLancien}, as well as in several research articles; see, e.g., \cite[Section~1]{ckk-mathematika}, \cite[p.~583, after Corollary~4.6]{PZ18}, \cite[Section~1]{CKK}, \cite[p.~2095]{AliagaPerneckaNormal22}, \cite[Section~1]{AliagaPernecka}, and \cite[pp.~1--2]{BGLMPS26}.

The finite-dimensional case was settled by the present authors jointly with Kaplick\'y~\cite{CKK}, who proved that $\F(E)$ is complemented in its bidual for every finite-dimensional normed space $E$, with a projection of norm one when $E$ is Euclidean. No example was known, however, of an infinite-dimensional Banach space $X$ for which $\F(X)$ is complemented in its bidual. The same paper explicitly asks whether $\F(\ell_2)$ is complemented in its bidual and whether a projection of norm one exists~\cite[Question~1.6]{CKK}.

A related line of research concerns the stronger property of $L$-embeddedness. Proch\'azka and Rueda Zoca~\cite[Corollary~4.6]{PZ18} showed that $\F(\ell_p)$ is not $L$-embedded for $1<p<\infty$. Aliaga, Perneck\'a, and Smith~\cite[Section~6]{APS25} constructed explicit $L$-projections for free spaces over
proper purely $1$-unrectifiable metric spaces. More recently, Braga, Gartland, Lancien, Motakis, Perneck\'a, and Schlumprecht~\cite{BGLMPS26} proved that $\F(X)$ is not $L$-embedded for any Banach space $X$ of dimension at least two.

One of our main results resolves the bidual complementability questions above by giving a complete  characterization for free spaces over separable Banach spaces and, more generally, over their open convex subsets.

\begin{thmx}\label{thm:Intro1}
Let $X$ be a separable Banach space and let $\Omega\subset X$
be a nonempty open convex subset.
\begin{enumerate}[$(a)$]
 \item $\F(\Omega)$ is complemented in its bidual if and
 only if $X$ is isomorphic to a Hilbert space.
 \item $\F(\Omega)$ is $1$-complemented in its bidual if
 and only if either $\dim X\le2$ or $X$ is isometric to a
 Hilbert space.
\end{enumerate}
\end{thmx}

In particular, $\F(\ell_1)$ is not complemented in its bidual, whereas $\F(\ell_2)$ is $1$-complemented in its bidual. The characterization also yields an isomorphic rigidity result: for every infinite-dimensional Banach space $Y$,
\[
 \F(Y)\cong\F(\ell_2)
 \quad\Longleftrightarrow\quad
 Y\cong\ell_2.
\]
Thus, among infinite-dimensional Banach spaces, the linear isomorphism class of $\ell_2$ is determined by that of its Lipschitz-free space. To the best of our knowledge, this is the first example of such a phenomenon.

In addition to the characterization above, we obtain several further results which we believe may be useful in the study of Lipschitz-free spaces over Banach spaces. We now describe our route to the main result and present several further results obtained along the way, which we believe may be useful in other contexts.

The first ingredient is the following representation theorem.

\begin{thmx}\label{thm:main}
Let $X$ be a nontrivial real separable Banach space and let $\Omega\subset X$ be an open convex set containing $0$. Let $\Gamma$ be a non-degenerate Gaussian probability measure on $X$. Denote by $\gamma$ the normalized restriction of $\Gamma$ to $\Omega$. For any $\g\in L^1(\gamma;X)$ 
define
$$S(\g)(f)=\int_\Omega \pair{f'(x)}{\g(x)}\,d\gamma(x),\quad f\in\Lipz(\Omega),$$
where $f'(x)$ denotes the G\^ateaux derivative of $f$ at $x$. 
Then $S$ is an isometric quotient of $L^1(\gamma;X)$ onto $\F(\Omega)$.
 Consequently, $\F(\Omega)$ is isometric to $L^1(\gamma;X)/\ker S$.
\end{thmx}

Let us point out that the assumption $0\in\Omega$ is not essential -- $0$ may be replaced by any other distinguished point $x\in \Omega$ and $\Lipz(\Omega)$ should be considered with respect to that distinguished point.

The above theorem may be viewed as an infinite-dimensional analogue of \cite[Theorem 1.1]{ckk-mathematika}. As in \cite{ckk-mathematika}, the adjoint operator of $S$ is the differentiation operator $f\mapsto f'$. But since the duality of Lebesgue-Bochner spaces is more complicated, the formulation requires some care. More precisely, $S^*:\Lipz(\Omega)\to (L^1(\gamma;X))^*$ is given by the formula
\begin{equation}\label{eq:D-main}
 (S^* f)(\g)=\int_\Omega \pair{f'(x)}{\g(x)}\,d\gamma(x)
 \qquad\g\in L^1(\gamma;X), f\in\Lipz(\Omega).
\end{equation}
This formula gives meaning to the claim that $S^*f=f'$. If $X^*$ is separable, then $(L^1(\gamma;X))^*=L^\infty(\gamma;X^*)$ and hence $S^*$ may be interpreted as an operator $\Lipz(\Omega)\to L^\infty(\gamma;X^*)$, similarly as in the finite-dimensional case. If $X^*$ is not separable, not only does the duality fail, but the mapping $f'$ need not be strongly measurable.

Theorem~\ref{thm:main} will be proved in Section~\ref{sec:proof-reprez} below. Then it is used to prove the main result, but we are convinced it has also an independent value. It is illustrated by easy proofs of two known results provided in Section~\ref{sec:first-consequences}.

The second ingredient is the `if' part of Theorem~\ref{thm:Intro1} formulated in the following result.

\begin{thmx}\label{thm:complement}
   Let $\Omega$ be a nonempty open convex subset of a separable Hilbert space. Then $\F(\Omega)$ is $1$-complemented in its bidual. 
\end{thmx}

This is an extension of \cite[Theorem 1.1]{CKK}. This extension is twofold -- the result includes also infinite-dimensional Hilbert spaces and, moreover, the projection of norm one exists not only for the case of the whole space, but for any nonempty convex subset. The proof follows the pattern of \cite{CKK} with several modifications and simplifications. 
One of the key ingredients of this proof is Theorem~\ref{thm:main}. Further, in Section~\ref{sec:norming-smooth} we introduce a $1$-norming subspace of $\Lipz(\Omega)$ consisting of smooth functions which enables us to use approximation by smooth functions similarly as in the finite-dimensional case. In Section~\ref{sec:bounded} we prove (using a similar pattern as in \cite{CKK}) the result for bounded $\Omega$. It is the content of Proposition~\ref{prop:dual-complementability-transfer}. In Section~\ref{sec:gluing} we give a general gluing principle for metric spaces -- Theorem~\ref{thm:fromBallToTheWholeSpace}. Theorem~\ref{thm:complement} then follows immediately by combining this theorem with Proposition~\ref{prop:compatible}.

The key tool to prove the `only if' part of Theorem~\ref{thm:Intro1} is the following theorem.

\begin{thmx}\label{thm:bidual-complemented-implies-Hilbert}
Let $X$ be a Banach space and let $M\subset X$ have nonempty
interior. If $\F(M)$ is $C$-complemented in its bidual for some
$C\geq1$, then every finite-dimensional subspace of $X$ is $C$-complemented in $X$.
\end{thmx}

This theorem will be proved in Section~\ref{sec:Mrep-complemented-implies-Hilbert} below. We note that, unlike in the previous theorems, here we do not assume that $X$ is separable. The reason is that a finite-dimensional reduction works (see Lemma~\ref{lem:projection-constant-local}). One of the important tools is again Theorem~\ref{thm:main}.

Let us now show how we get our main result by combining the other theorems.

\begin{proof}[Proof of Theorem~\ref{thm:Intro1}]
$(a)$: The `if' part follows immediately from Theorem~\ref{thm:complement}. To prove the `only if' part assume that $\F(\Omega)$ is complemented in $\F(\Omega)^{**}$. Let $C\ge 1$ be the norm of a witnessing projection. By Theorem~\ref{thm:bidual-complemented-implies-Hilbert} each finite-dimensional subspace of $X$ is $C$-complemented in $X$.  The Lindenstrauss--Tzafriri characterization of Hilbert spaces \cite{LindenstraussTzafriri1971}, see also
\cite[Theorem 13.4.3]{AlbiacKalton}, implies that $X$ is isomorphic to a Hilbert space.

$(b)$: Let us start with the `if' part: If $X$ is isometric to a Hilbert space, we conclude by Theorem~\ref{thm:complement}. If $\dim X\le 2$, by the very recent result \cite[Corollary 4.3]{BGLMPS26b} we obtain that $\F(X)$ is $1$-complemented in its bidual, which proves the assertion for $\Omega=X$. The case of general $\Omega$ may be easily deduced by combining the approach of \cite{BGLMPS26b} with the method of proof of our Theorem~\ref{thm:complement}, as explained in Remarks~\ref{rem:dim-bounded} and~\ref{rem:twodim-compatible} below.

To prove the `only if' part assume that $\F(\Omega)$ is $1$-complemented in $\F(\Omega)^{**}$ and $\dim X\geq3$. Let $E\subset X$
be three-dimensional. By Theorem~\ref{thm:bidual-complemented-implies-Hilbert}
and restriction, any subspace of $E$ is $1$-complemented in $E$.
Kakutani's theorem \cite[Theorem~4]{Kakutani}, see also \cite[pp.~98--106]{Amir1986}, therefore makes $E$ Euclidean.
Every pair of vectors in $X$ lies in such an $E$, so the norm
of $X$ satisfies the parallelogram identity. Hence $X$ is a Hilbert space.  
\end{proof}

\begin{remark}
    (1) We point out that the `only if' parts in Theorem~\ref{thm:Intro1} do not require the separability assumption. The reason is that in Theorem~\ref{thm:bidual-complemented-implies-Hilbert} this assumption is not needed. However, in the `if' parts separability is strongly used, in particular due to the use of Theorem~\ref{thm:main}. We have not investigated the problem for nonseparable Hilbert spaces.

    (2) In the `only if' part of assertion $(b)$ of Theorem~\ref{thm:Intro1} it is enough to assume that $\F(\Omega)$ is $(1+\varepsilon)$-complemented in $\F(\Omega)^{**}$ for each $\varepsilon>0$. Indeed, given $E\subset X$ of dimension $3$, we get that each subspace is $1+\varepsilon$-complemented for each $\varepsilon>0$. By compactness we deduce that each subspace of $E$ is $1$-complemented and hence we may conclude as above.
\end{remark}

\section{Preliminaries}

All Banach spaces and scalar-valued functions in this paper are real. For a metric space $M$, we denote by $U_M(x,r)$ and $B_M(x,r)$ the open and closed balls with centre $x$ and radius $r>0$, omitting the subscript when the ambient space is clear. For a Banach space $X$, its open and closed unit balls are denoted by $U_X$ and $B_X$, respectively.

Recall that a linear subspace $W\subset X^*$ is called \emph{$1$-norming} for $X$ if
\[
 \norm{x}=\sup_{w\in B_W}|w(x)|,\qquad x\in X.
\]
It is a standard consequence of the bipolar theorem that $W$ is $1$-norming for $X$ if and only if $B_{X^*}$ is the weak$^*$ closure of $B_W=W\cap B_{X^*}$.

We use the standard notions of directional and G\^ateaux
differentiability; see \cite[Definitions~1.65--1.66]{HajekJohanis}. If $U\subset X$ is open, $f:U\to\mathbb R$, $x\in U$, and $h\in X$, we write
\[
 \partial_h f(x)=\lim_{t\to0}\frac{f(x+th)-f(x)}{t},
\]
whenever the limit exists. If $f$ is G\^ateaux differentiable at $x$, its derivative is denoted by $f'(x)\in X^*$, so that $\partial_h f(x)=f'(x)[h]$. In particular, every $x^*\in X^*$ is G\^ateaux differentiable at each $x\in X$, with $(x^*)'(x)=x^*$; equivalently, $\partial_h x^*(x)=x^*(h)$ for every $h\in X$. If $f$ is $L$-Lipschitz, then obviously $\norm{f'(x)}\le L$ at every point of G\^ateaux differentiability.

For a Banach space $Y$ and an open set $U\subset X$, we denote by $C^\infty(U;Y)$ the class of mappings $F:U\to Y$ whose
Fr\'echet derivatives of every order exist and are continuous in the corresponding operator norms; see \cite[Definitions~1.66 and~1.74]{HajekJohanis}. Here the Fr\'echet derivative $F'(x):X\to Y$ is the bounded
linear operator satisfying $\norm{F(x+h)-F(x)-F'(x)h}=o(\norm h)$ as $h\to0$, and higher derivatives are defined inductively. If $F$ is Fr\'echet differentiable at $x$, then it is also G\^ateaux differentiable at $x$, and the two derivatives coincide. Thus our notation is consistent.

\subsection{Lipschitz-free spaces}\label{subsec:lipschitz-free}

Let $(M,d,0)$ be a pointed metric space, that is, a metric space
with a distinguished point denoted by $0$. A map between pointed
metric spaces is called \emph{pointed} if it maps the distinguished
point of its domain to that of its target space.

For a Lipschitz map $f:M\to Y$, where $Y$ is a Banach space,
we denote its least Lipschitz constant by
\[
 \norm{f}_{\Lip}
 =\sup_{\substack{x,y\in M\\x\ne y}}
   \frac{\norm{f(x)-f(y)}}{d(x,y)}.
\]
This defines a seminorm on the vector space of Lipschitz maps
from $M$ to $Y$, and a norm on its subspace of pointed maps.
By $\Lipz(M)$ we denote the Banach space of pointed real-valued
Lipschitz functions on $M$, equipped with this norm.

For $x\in M$, denote by $\delta_M(x)\in\Lipz(M)^*$ the evaluation
functional given by $\pair{\delta_M(x)}f=f(x)$ for $f\in\Lipz(M)$. The
\emph{Lipschitz-free space} over $M$, denoted by $\F(M)$, is the
closed linear span of $\delta_M(M)$ in $\Lipz(M)^*$, equipped with
the inherited norm. Observe that $\delta_M:M\to\F(M)$ is a pointed
isometric embedding.

The space $\F(M)$ is characterized by the following
\emph{universal property}: whenever $Y$ is a Banach space and
$f:M\to Y$ is a pointed Lipschitz map, there is a unique bounded
linear operator $\linearization{f}:\F(M)\to Y$ such that
$\linearization{f}\circ\delta_M=f$.
Moreover, $\norm{\linearization{f}}=\norm{f}_{\Lip}$.
The operator $\linearization{f}$ is called the \emph{linearization}
of $f$.
\begin{center}
\begin{tikzcd}[column sep=large,row sep=large]
 M \arrow[r,"f"] \arrow[d,"\delta_M"']
   & Y \\
 \F(M) \arrow[ur,dashed,"\linearization{f}"']
   &
\end{tikzcd}
\end{center}
Applying the universal property with $Y=\R$, we obtain
the canonical isometric identification $\F(M)^*=\Lipz(M)$.
With respect to this duality, the weak-star topology on bounded
subsets of $\Lipz(M)$ coincides with the topology of pointwise convergence on $M$.

Further, whenever $0\in N\subset M$, we regard $\F(N)$ as the closed subspace of $\F(M)$ generated by $\delta_M(N)$; this identification is justified by McShane's extension theorem. By continuity of $\delta_M$, we have $\F(N)=\F(M)$ if $N$ is dense in $M$. In particular, $\F(M)$ is separable whenever $M$
is separable.

The facts recalled above are standard; for the convenience of the reader, we refer to \cite[Section~1]{cuth-doucha-w} and to the monograph of Baudier and Lancien \cite[Section~4.1]{BaudierLancien} for proofs and further background. Our notation follows the latter, with minor modifications. For later reference, we also record a few elementary observations together with their short proofs.

\begin{fact}\label{fact:linearization-adjoint}
Let $M$ be a pointed metric space and let $f:M\to Y$ be a pointed
Lipschitz map into a Banach space. Then
$\pair{\linearization{f}(\xi)}{y^*}
=\pair{\xi}{(y^*\circ f)|_N}$
whenever $0\in N\subset M$, $\xi\in\F(N)$ and $y^*\in Y^*$.
\end{fact}

\begin{proof}
For $\xi=\delta_M(x)$ with $x\in N$, both sides equal $y^*(f(x))$.
The assertion follows by linearity and continuity.
\end{proof}

The following standard result concerns weighting operators on
Lipschitz-free spaces; for the convenience of the reader, we refer e.g. to \cite[Lemma~2.3 and Proposition~2.4]{Aliaga-supports} for a proof.

\begin{lemma}\label{lem:free-weighting}
Let $(M,d,0)$ be a pointed metric space and let $g:M\to\R$ be
a bounded Lipschitz function with
$\supp g\subset\{x\in M:d(x,0)\le R\}$ for some $R\ge0$. Then there is a unique bounded linear operator
$W_g:\F(M)\to\F(M)$ satisfying
$W_g\delta_M(x)=g(x)\delta_M(x)$ for every $x\in M$.

Moreover, $\norm{W_g}\le\norm{g}_\infty+R\norm{g}_{\Lip}$,
$W_g^*f=gf$ for every $f\in\Lipz(M)$, and
$W_g(\F(M))\subset\F(\{0\}\cup\supp g)$.
\end{lemma}

\subsection{Approximation by convolution}
\label{subsec:convolution}

Let $E=(\R^d,\norm{\cdot})$, where $d\ge1$ and $\norm{\cdot}$
is an arbitrary norm. All integrals below are taken with respect to the usual Lebesgue measure on $\R^d$.
Following \cite[Section~2]{ckk-mathematika} and
\cite[Section~2]{CKK}, we fix an \emph{approximate unit}
$(u_n)_{n\in\N}$ in $C_c^\infty(\R^d)$, given by
$u_n(x)=n^d\rho(nx)$, where $\rho\in C_c^\infty(\R^d)$ is
nonnegative, $\int_{\R^d}\rho(x)\,dx=1$, and
$\supp\rho\subset U_E$. Thus $\supp u_n\subset n^{-1}U_E$ and
$\int_{\R^d}u_n(x)\,dx=1$ for every $n\in\N$.

For a locally integrable function $f:E\to\R$, define
\[
 (f*u_n)(x)=\int_{\R^d}f(x-z)u_n(z)\,dz,
 \qquad x\in E,\ n\in\N.
\]
These integrals are well defined since $u_n$ is bounded
and compactly supported. Recall that convolution is
commutative, so $f*u_n=u_n*f$.
For finite-dimensional vector-valued functions, such as
the derivative $f'$, convolution is understood coordinatewise.

The following approximation properties are standard.
For the convenience of the reader, we  include a short sketch of the proof in the form needed here.

\begin{fact}\label{fact:convolution-approximation}
Let $E=(\R^d,\norm{\cdot})$, let $(u_n)$ be an approximate
unit as above, and let $f:E\to\R$ be Lipschitz.
For every $n\in\N$, we have $u_n*f\in C^\infty(E)$ and
\begin{equation}\label{eq:convolution-derivative}
 (u_n*f)'(x)=(u_n*f')(x),\qquad x\in E.
\end{equation}
Moreover,
\[
 \norm{u_n*f}_{\Lip}
 =\norm{u_n*f'}_\infty
 \le\norm f_{\Lip},
 \qquad
 \norm{u_n*f-f}_\infty
 \le n^{-1}\norm f_{\Lip}.
\]
\end{fact}

\begin{proof}
Fix $n\in\N$. The smoothness of $u_n*f$ follows by
standard differentiation under the integral sign;
see, e.g., \cite[Corollary~1.91]{HajekJohanis}.

By Rademacher's theorem, $f'$ exists almost everywhere
and satisfies $\norm{f'(y)}\le\norm f_{\Lip}$ wherever
it exists. We set $f'=0$ at the remaining points.
Fix $x,h\in E$. By dominated convergence,
\[
 \begin{aligned}
 (u_n*f)'(x)[h]
 &=\lim_{\tau\to0}\int_E u_n(z)
   \frac{f(x+\tau h-z)-f(x-z)}{\tau}\,dz\\
 &=\int_E u_n(z)f'(x-z)[h]\,dz
 =(u_n*f')(x)[h].
 \end{aligned}
\]
Indeed, the difference quotients converge for almost
every $z$ and are bounded in absolute value by
$\norm f_{\Lip}\norm h$.
Since $h$ was arbitrary, this proves
\eqref{eq:convolution-derivative}.

Since $u_n*f$ is smooth and $E$ is convex, we have
\[
 \norm{u_n*f}_{\Lip}
 =\sup_{x\in E}\norm{(u_n*f)'(x)}
 =\norm{u_n*f'}_\infty
 \le\norm f_{\Lip},
\]
where the last inequality follows from
$\norm{f'}_\infty\le\norm f_{\Lip}$, $u_n\ge0$,
and $\int_Eu_n=1$.
Finally, for every $x\in E$,
\[
 \begin{aligned}
 |(u_n*f)(x)-f(x)|
 &\le\int_E u_n(z)|f(x-z)-f(x)|\,dz\\
 &\le\norm f_{\Lip}\int_E u_n(z)\norm z\,dz
 \le n^{-1}\norm f_{\Lip},
 \end{aligned}
\]
because $\supp u_n\subset n^{-1}U_E$.
Taking the supremum over $x\in E$ completes the proof.
\end{proof}

\subsection{Subspaces and quotients}

In this section we fix terminology concerning quotient operators and recall the standard duality of subspaces and quotients.  

Let $X,Y$ be normed spaces and let $T:X\to Y$ be a bounded linear operator. If $T$ preserves the norm (i.e., $\norm{Tx}=\norm{x}$ for each $x\in X$), $T$ is called an \emph{isometric embedding}.
Further, $T$ is called an \emph{isometric quotient} if $T$ maps the open unit ball in $X$ onto the open unit ball in $Y$. This is easily seen to be equivalent to the fact that the induced operator
$X/\ker T\to Y$ is a surjective isometry. We emphasize that we do not assume that both $X$ and $Y$ are Banach spaces. This additional generality will be useful later and should be kept in mind.

\begin{proposition}\label{prop:subspaces-quotients}
    Let $X,Y$ be normed spaces and let $T:X\to Y$ be a bounded linear operator. Then the following assertions are valid.
    \begin{enumerate}[label=(\roman*)]
        \item\label{item:isometric-embedding} $T$ is an isometric embedding if and only if $T^*$ is an isometric quotient.
        \item\label{item:quotient-dual-embedding} If $T$ is an isometric quotient, then $T^*$ is an isometric embedding.
        \item\label{item:complete-quotient} If $X,Y$ are additionally complete, then $T$ is an isometric quotient if and only if $T^*$ is an isometric embedding.
    \end{enumerate}
\end{proposition}

\begin{proof}
    This proposition is completely standard. We sketch the proof for the sake of completeness. The `only if' part of \ref{item:isometric-embedding} is an easy consequence of Hahn-Banach extension theorem. Assertion \ref{item:quotient-dual-embedding} follows easily from the definition of the dual norm. 
    
    Let us prove the `if' part of \ref{item:isometric-embedding}: Assume $T^*$ is an isometric quotient. By \ref{item:quotient-dual-embedding} we know that $T^{**}$ is an isometric embedding. Hence $T$, being the restriction of $T^{**}$ to $X$ is also an isometric embedding.

    \ref{item:complete-quotient}: The `only if' part follows from $(ii)$ and does not require completeness. The `if' part follows from \cite[Theorem 4.13]{Rudin}, implication $(a)\implies(c)$, applied to $\delta=1$.
\end{proof}

\subsection{Gaussian measures}
In the study of Lipschitz-free spaces over subsets of $\R^d$ in \cite{ckk-mathematika,CKK} an important role was played by the Lebesgue measure on $\R^d$. On infinite-dimensional spaces there is no canonical measure of this kind. However, Gaussian measures may be used instead.
We recall briefly the relevant terminology and properties.  For $m\in\R$ and
$\sigma>0$, the Gaussian probability $N(m,\sigma^2)$ on $\R$ is the
probability measure with density
\[
 dN(m,\sigma^2)(t)
 =\frac{1}{\sqrt{2\pi\sigma^2}}
  \exp\!\left(-\frac{(t-m)^2}{2\sigma^2}\right)\,dt.
\]
We also put $N(m,0)=\delta_m$, the Dirac measure supported at $m$. 

Let $X$ be a real separable Banach space.  A Borel probability
$\Gamma$ on $X$ is called \emph{Gaussian} if, for every $x^*\in X^*$,
there are $m_{x^*}\in\R$ and $\sigma_{x^*}\geq0$ such that
\[
 x^*_\#\Gamma=N(m_{x^*},\sigma_{x^*}^2);
\]
see \cite[Definition~6.17]{BenyaminiLindenstrauss}.  It is called
\emph{nondegenerate} if $\sigma_{x^*}>0$ whenever $x^*\ne0$.

\begin{fact}
    Any nontrivial separable Banach space admits a nondegenerate Gaussian probability measure.
\end{fact}

\begin{proof}
     For a separable Hilbert space this follows, for example, from \cite[Theorem~6.16 and Proposition~6.18 together with Remark (i) below Proposition~6.18]{BenyaminiLindenstrauss}.
     Indeed, the quoted results show how to construct such a measure starting from any injective positive trace-class operator on the Hilbert space. Further, let $X$ be a nontrivial separable Banach space. 
     Then it is easy to find a bounded linear operator $T:\ell^2\to X$ with dense range (see, e.g., \cite[p. 140]{BenyaminiLindenstrauss}). Let $\mu$ be a non-degenerate Gaussian probability on $\ell^2$ (it exists by the above). Then its pushforward $T_\#\mu$ is a nondegenerate Gaussian probability on $X$ by \cite[Proposition 6.20]{BenyaminiLindenstrauss}.
\end{proof}

\begin{proposition}\label{prop:gaussian}
    Let $X$ be a nontrivial separable Banach space and let $\Gamma$ be a non-degenerate Gaussian probability on $X$. Then the following assertions are valid.
    \begin{itemize}
\nameditem{FS}{FS} $\Gamma$ has full support, i.e., $\Gamma(U)>0$ for each nonempty open set $U\subset X$.
\nameditem{D}{D}
For every Lipschitz function $f:X\to\R$, there is a Borel
set $N_f\subset X$ with $\Gamma(N_f)=0$ such that $f$ is G\^ateaux differentiable at every
point of $X\setminus N_f$.

\nameditem{NT}{NT}
There is a norm-dense linear subspace $H_\Gamma\subset X$ such that, for every $h\in H_\Gamma$, every $t\in\mathbb R$, and every $\Gamma$-measurable set $N\subset X$
\begin{equation*}
   \Gamma(N)=0
   \quad\Longrightarrow\quad
   \Gamma(N-th)=0.
\end{equation*}

\nameditem{GIP}{GIP}
For the same space $H_\Gamma$, there is a linear map
\[
 h\longmapsto b_h\quad\text{from }H_\Gamma\text{ into }L^1(\Gamma)
\]
such that, for every $h\in H_\Gamma$, every Lipschitz $f:X\to\R$, and every
bounded Lipschitz $\psi:X\to\R$ with bounded support,
\begin{equation*}
 \int_X\partial_h f(x)\cdot \psi(x)\,d\Gamma(x)
 =\int_X f(x)\cdot\bigl(b_h(x)\psi(x)-\partial_h\psi(x)\bigr)
   \,d\Gamma(x).
\end{equation*}
\end{itemize}
\end{proposition}

\begin{proof}
\ref{FS}: By \cite[Theorem~7.1.7]{BogachevMeasureTheory} we deduce that $\Gamma$ is a Radon measure.  Thus we may use  \cite[Theorem~3.6.1]{Bogachev} to show that the support of $\Gamma$ is a closed affine subspace of $X$. It now easily follows that $\Gamma$ has full support if and only if it is nondegenerate.

\ref{D}: This is the statement of Phelps' differentiability theorem, see
\cite[Theorem~6]{Phelps}, or \cite[Proposition~6.25 and Theorem~6.42]{BenyaminiLindenstrauss}.

\ref{NT}: Let $H_\Gamma$ be the Cameron--Martin space of $\Gamma$; see
\cite[Notation on p.~44]{Bogachev}.  If $h\in X$, write $\Gamma_h$ for the translate
of $\Gamma$ by $h$.  The Cameron--Martin theorem  \cite[Theorem~2.4.5]{Bogachev} says that
$\Gamma_h$ and $\Gamma$ are mutually absolutely continuous for $h\in H_\Gamma$ and mutually
singular for $h\notin H_\Gamma$.  The space $H_\Gamma$ is linear by definition, and it is
norm dense in $X$ because $\Gamma$ has full support (as by
\cite[Theorem~3.6.1]{Bogachev} the support of $\Gamma$ is a translate of $\overline{H_\Gamma}$). Since $H_\Gamma$ is a linear space, the implication from \ref{NT} follows.

\ref{GIP}: Since $\Gamma$ is a Gaussian measure on $X$, clearly $X^*\subset L^2(\Gamma)$.
Following \cite[p.~44]{Bogachev}, we denote by $X_\Gamma^*$ the closure in $L^2(\Gamma)$ of
\[
 \mathcal{L}_\Gamma:=\left\{x^*-\int_Xx^*\,d\Gamma:x^*\in X^*\right\}.
\]
Being a closed subspace of a Hilbert space, it is again a Hilbert space, with the inherited inner product. For $g\in X_\Gamma^*$ define a linear functional $R_\Gamma g$ on $X^*$
by
\[
 (R_\Gamma g)(x^*)
 =\int_X\left(x^*(x)-\int_Xx^*\,d\Gamma\right)g(x)\,d\Gamma(x),\qquad x^*\in X^*.
\]
The operator $R_\Gamma$ is one-to-one.  Indeed, suppose that
$g\in X_\Gamma^*$ and $R_\Gamma g=0$. Then $g\perp \mathcal{L}_\Gamma$  in the Hilbert space $X_\Gamma^*$. Since $\mathcal{L}_\Gamma$ is dense in $X_\Gamma^*$, we deduce $g=0$.

By \cite[Lemma~2.4.1]{Bogachev}, a vector $h\in X$ belongs to
$H_\Gamma$ if and only if there is $g\in X_\Gamma^*$ such that
$R_\Gamma g=h$, that is,
\[
x^*(h) = \int_X\left(x^*(x)-\int_Xx^*\,d\Gamma\right)g(x)\,d\Gamma(x),\qquad x^*\in X^*.
\]
Such a $g$ is unique due to injectivity of $R_\Gamma$.  Following \cite[Section~2.4]{Bogachev}, we denote this unique
element by $\widehat h$.  The map $h\mapsto\widehat h$ is linear.  Put
$b_h=\widehat h$ for $h\in H_\Gamma$; then
$b_h\in L^2(\Gamma)\subset L^1(\Gamma)$.

Let $f:X\to\R$ be Lipschitz and let $\psi:X\to\R$ be bounded and
Lipschitz with bounded support.  Put $g=f\psi$.  Then $g$ has bounded
support and is bounded.  It is also globally Lipschitz.  Indeed, choose
$r>0$ such that $\supp\psi\subset B(0,r)$.  The product $f\psi$ is
Lipschitz on $B(0,2r)$, and it vanishes outside $B(0,r)$.  If
$x\in B(0,r)$ and $y\notin B(0,2r)$, then
$|g(x)-g(y)|\leq\norm g_\infty r^{-1}\norm{x-y}$, which proves the
remaining case.

Fix $h\in H_\Gamma$.  The function $g$ satisfies the assumptions of
\cite[Theorem~5.1.8]{Bogachev}: it is measurable; for every $x\in X$,
the function $t\mapsto g(x+th)$ is Lipschitz and hence absolutely
continuous; $\partial_hg$ is bounded wherever it exists and therefore integrable;
and $g\widehat h\in L^1(\Gamma)$ because $g$ is bounded and
$\widehat h\in L^2(\Gamma)\subset L^1(\Gamma)$.  Consequently, by \cite[Theorem~5.1.8]{Bogachev} we obtain
\[
 \int_X\partial_hg\,d\Gamma=\int_Xg\widehat h\,d\Gamma.
\]
Choose a Borel $\Gamma$-null set outside which both $f$ and $\psi$ are
G\^ateaux differentiable.  At every point outside this set, the product
rule gives $\partial_hg=\partial_hf\,\psi+f \partial_h\psi$.  Hence
\[
 \int_X\bigl(\partial_hf\,\psi+f \partial_h\psi\bigr)\,d\Gamma
 =\int_Xf\psi\widehat h\,d\Gamma,
\]
and therefore
\[
 \int_X\partial_hf\,\psi\,d\Gamma
 =\int_Xf\bigl(\psi b_h-\partial_h\psi\bigr)\,d\Gamma.
\]
This completes the proof.
\end{proof}

\begin{remark}\label{rem:gauss}
  (1) In the sequel we will use only the four properties of non-degenerate Gaussian measures collected in Proposition~\ref{prop:gaussian}. Hence, in all results in this paper which deal with Gaussian measures, the assumption that $\Gamma$ is a non-degenerate Gaussian measure may be replaced by the assumption that $\Gamma$ is a Borel probability satisfying the four properties.

(2) We point out that it follows from property \ref{D} that $\Gamma(\partial \Omega)=0$ whenever $\Omega$ is an open convex set. Indeed, if $\Omega\subset X$ is open and convex, the function $f(x)=\dist(x,\Omega)$ is $1$-Lipschitz and each boundary point of $\Omega$ is easily seen to be a point of G\^ateaux non-differentiability. So, by (D) we deduce that it is $\Gamma$-null. Therefore, we may identify restrictions $\Gamma|_\Omega$ and $\Gamma|_{\overline{\Omega}}$ in the obvious way. We will do it in the sequel. 

  (3) Note that if $X$ has finite dimension, we may consider also the Lebesgue measure. It is only $\sigma$-finite (which causes no difficulty at least on bounded domains), but it satisfies the four properties. Indeed, it has full support and property \ref{D} follows from the Rademacher theorem.
  Property \ref{NT} with $H_\Gamma=X$ follows from the translation invariance of Lebesgue measure. Finally, property \ref{GIP} with $b_h$ being constant zero is the standard integration by parts for Lipschitz functions. 
\end{remark}

\subsection{Bochner integration}

We use the standard terminology from \cite[Chapter II]{DiestelUhl} or \cite[Section 2.3]{ryan}. We work only with finite measures. 
In particular, if $(\Omega,\Sigma,\mu)$ is a finite measure space and $X$ is a Banach space, by $L^1(\mu;X)$ we denote the respective Lebesgue-Bochner space (see \cite[pp. 49--50]{DiestelUhl}
or \cite[p. 29]{ryan}). If $X^*$ has the Radon-Nikod\'ym property, the dual $(L^1(\mu;X))^*$ is canonically isometric to $L^\infty(\mu;X^*)$ (see, e.g., \cite[Theorem IV.1.1]{DiestelUhl}). 
If $X^*$ fails the Radon-Nikod\'ym property, the description of the dual is complicated. However, there is a nice natural subspace of the dual which we will use below.

Denote by $L^\infty_{w^*}(\mu;X^*)$ the space of all (equivalence classes of) essentially bounded weak$^*$-measurable mappings $\g:\Omega\to X^*$, equipped with the essential supremum norm. Recall that $\g$ is \emph{weak$^*$-measurable} if $\omega\mapsto \g(\omega)(x)$ is measurable for each $x\in X$. It is easy to check that in this way we obtain a Banach space.

\begin{lemma}\label{lem:field-functional}
Let $(\Omega,\Sigma,\mu)$ be a finite measure space and let $X$ be a nonzero
separable Banach space. Then $L^\infty_{w^*}(\mu;X^*)$ is canonically isometric to a subspace of $(L^1(\mu;X))^*$ via the mapping
\[
\Phi(\g)(\f)=\int_\Omega \pair{\g(\omega)}{\f(\omega)}\,d\mu(\omega),\quad \f\in L^1(\mu;X), \g\in L^\infty_{w^*}(\mu;X^*).
\]
\end{lemma}

\begin{proof} The proof is done in three steps. The first two steps may be done essentially by copying the argument from \cite[pp. 97--98]{DiestelUhl}. We give the easy arguments because we have a slightly different setting.

Step 1: If $\f\in L^1(\mu;X)$ and $\g\in L^\infty_{w^*}(\mu;X^*)$, then the mapping $\omega\mapsto \pair{\g(\omega)}{\f(\omega)}$ is measurable.

\smallskip

Indeed, since $\f$ is strongly measurable, there
are simple measurable functions $\f_n\colon \Omega\to X$ such that
$\f_n\to \f$ almost everywhere. For every $n$, the function
$\omega\mapsto \pair{\g(\omega)}{\f_n(\omega)}$ is measurable by the weak$^*$-measurability of $\g$. Furthermore,
\[\pair{\g(\omega)}{\f_n(\omega)}\to \pair{\g(\omega)}{\f(\omega)}\mbox{ a.e.}
\]
hence the limit function is measurable as well.

\smallskip

Step 2: $\Phi$ is a contractive mapping from $L^\infty_{w^*}(\mu;X^*)$ into $(L^1(\mu;X))^*$.

\smallskip

Given $\g\in L^\infty_{w^*}(\mu;X^*)$ and $\f\in L^1(\mu;X)$, the mapping $\omega\mapsto \pair{\g(\omega)}{\f(\omega)}$ is measurable by Step 1. It is also integrable because
$\abs{\pair{\g(\omega)}{\f(\omega)}}\le\norm{\g(\omega)}\norm{\f(\omega)}\le \norm{\g}_\infty\norm{\f(\omega)}$ almost everywhere. We now easily deduce that $\Phi(\g)\in (L^1(\mu;X))^*$ and $\norm{\Phi(\g)}\le\norm{\g}_\infty$. Since $\Phi$ is clearly linear, the argument is complete.

\smallskip

Step 3: $\norm{\Phi(\g)}\ge\norm{\g}_\infty$, hence $\Phi$ is an isometry.

\smallskip
If $\g=0$, the statement is trivial. So, assume $\g\ne 0$. Fix $0<a<\norm{\g}_\infty$ and let
$E=\{\omega:\norm{\g(\omega)}>a\}$, which has positive measure.  Choose a countable
dense set $(u_j)$ in the unit sphere of $X$.  For every $\omega\in E$, there is some $j$ such that
$\pair{\g(\omega)}{u_j}>a$. Define
\[
 E_j
 :=
 \bigl\{\omega\in E:\pair{\g(\omega)}{u_j}>a
 \text{ and }\pair{\g(\omega)}{u_i}\leq a\text{ for every }i<j\bigr\}.
\]
Then $(E_j)_{j\geq1}$ is a measurable partition of $E$. 
Define
\[
 \f(\omega)=  \frac{u_j}{\mu(E)},\quad \omega\in E_j, j\in\N.
\]
The function $\f$ is countably valued and strongly measurable,
$\|\f\|_1=1$, and 
$$\pair{\g(\omega)}{\f(\omega)}>\frac{a}{\mu(E)} \mbox{ for every }\omega\in E.$$  Thus
$\Phi(\g)(\f)>a$. This completes the proof.
\end{proof}

The previous lemma is inspired by the isometric inclusion $L^\infty(\mu;X^*)\subset (L^1(\mu;X))^*$ proved in \cite[pp. 97--98]{DiestelUhl}. The first two steps are essentially the same. But the third step is different. In \cite{DiestelUhl} it is used that elements of $L^\infty(\mu;X^*)$ are (by definition) strongly measurable; we instead impose and use separability of $X$. Our lemma fails without the assumption of separability of $X$. A counterexample may be easily constructed starting from the standard example \cite[Example II.1.5]{DiestelUhl}.

\subsection{Integration with respect to certain vector measures}\label{ssec:vectormeasures}

We will further use integration of vector valued functions with respect to measures with values in the dual space. The basic setting is the following. Let $(\Omega,\Sigma)$ be a measurable space and let $X$ be a Banach space. Let $\bmu:\Sigma\to X^*$ be a $\sigma$-additive regular measure with finite total variation. Then we may define in a natural way $\int_\Omega \f\,d\bmu$ for any bounded function $\f:\Omega\to X$ which may be uniformly approximated by simple measurable functions. This is explained in more detail, for example, in \cite[Section 2.3]{KS-transference}.

This type of integral is used in the Riesz-Singer representation theorem. Let us recall it. Let $K$ be a compact Hausdorff space and let $X$ be a nontrivial Banach space. Consider the space $C(K,X)$ of all $X$-valued continuous maps on $K$ equipped with the sup-norm. Its dual is canonically isometric to $M(K,X^*)$, the space of all $X^*$-valued  $\sigma$-additive regular Borel measures on $K$ with finite total variation, equipped with the total variation norm. This is a result of Singer; an easy proof is given in \cite{hensgen}. The duality is given by
$$\pair{\bmu}{\f}=\int_K\f\,d\bmu, \quad \f\in C(K,X),\bmu\in M(K,X^*),$$
where the integral is in the sense of the previous paragraph. (Note that continuous functions $K\to X$ may be uniformly approximated by simple Borel functions.)

If, additionally, $X^*$ is separable (or, more generally, $X^*$ has the Radon-Nikod\'ym property), 
then any $\bmu\in M(K,X^*)$ has a density with respect to its variation $\abs{\bmu}$. I.e., there is a function $\uu\in L^1(\abs{\bmu};X^*)$ such that
$$\bmu(A)=\int_A \uu\,d\abs{\bmu},\quad A\subset K\mbox{ Borel}.$$
It is known and easy to see that necessarily $\norm{\uu}=1$ $\abs{\bmu}$-a.e. In this setting we have
\begin{equation}\label{eq:dual-vector-pairing}
 \pair{\bmu}{\f}  =\int_K\pair{\uu(t)}{\f(t)}\,d|\boldsymbol\mu|(t),\quad \f\in C(K,X).
\end{equation}
Indeed, this formula clearly holds for simple Borel functions and hence it passes to all continuous functions (using the uniform approximation by simple Borel functions).

\section{Proof of the representation theorem}\label{sec:proof-reprez}

The aim of this section is to prove Theorem~\ref{thm:main}. It will be done in several steps which partially follow the finite-dimensional case from \cite{ckk-mathematika}. So, let us fix a nontrivial separable Banach space $X$, an open convex subset $\Omega\subset X$ containing $0$ and a non-degenerate Gaussian measure $\Gamma$ on $X$. As in the theorem, denote by $\gamma$ the normalized restriction of $\Gamma$ to $\Omega$.

\smallskip

{\bf Step 1.} The mapping $D:\Lipz(\Omega)\to L^\infty_{w^*}(\gamma,X^*)$ defined by $Df=f'$ is a linear isometric embedding. (By $f'$ we mean the mapping assigning to $x\in\Omega$ the G\^ateaux derivative $f'(x)$.)

\smallskip

By property \ref{D} $f$ is $\gamma$-a.e. G\^ateaux differentiable (note that \ref{D} formally deals with Lipschitz functions on $X$, but this causes no difficulty as any Lipschitz function on $\Omega$ may be extended to a Lipschitz function on $X$). Thus the mapping $f'$ is defined $\gamma$-a.e. Moreover,
for each $z\in X$ the assignment
$$x\mapsto Df(x)(z)=f'(x)[z]=\partial_z f(x)$$
is Baire-one (on its domain), hence $\gamma$-measurable. Thus $f'$ is weak$^*$-measurable. Moreover, by definition it is clear that $\norm{f'}_\infty\le \norm{f}_{\mathrm{Lip}}$. Hence $D$ is a well-defined mapping from $\Lipz(\Omega)$ to $L^\infty_{w^*}(\gamma,X^*)$. Since it is clearly linear, we get $\norm{D}\le1$.

It remains to prove that $D$ is an isometry. To this end fix $f\in\Lipz(\Omega)$ and set
$L=\norm{Df}_\infty$. Let $F:X\to \R$ be a Lipschitz extension of $f$. By property \ref{D} we may fix a Borel $\Gamma$-null set $N_F\subset X$ such that $F$ is G\^ateaux differentiable outside $N_F$.
The set
\[
 N=(N_F\cap\Omega)\cup\{x\in\Omega:\norm{f'(x)}>L\}
\]
is $\Gamma$-null. Fix $h\in H_\Gamma$.  Property \ref{NT} gives
$\Gamma(N-th)=0$ for every $t\in\R$.  Hence Fubini-Tonelli's theorem yields
\[\begin{aligned}
     \int_X\mathcal L^1\{t\in[0,1]:x+th\in N\}\,d\Gamma(x)&=\int_X\int_0^1 \chi_{\{(t,h): x+th\in N\}} \, dt\, d\Gamma(x)\\& =\int_0^1\Gamma(N-th)\,dt=0,\end{aligned}
\]
where $\mathcal{L}^1$ stands for the Lebesgue measure on $\R$.
Thus, for $\Gamma$-almost every $x$, the segment $x+[0,1]h$ meets $N$ in
a one-dimensional null set.

Put $V_h=\Omega\cap(\Omega-h)$.  This set is open, and convexity of
$\Omega$ gives $x+[0,1]h\subset\Omega$ whenever $x\in V_h$.  For every
such $x$, the function $t\mapsto f(x+th)$ is Lipschitz and hence absolutely
continuous.  For $\Gamma$-almost every $x\in V_h$, it satisfies, for almost
every $t\in[0,1]$,
\[
 \frac{d}{dt}f(x+th)=Df(x+th)[h],
 \qquad
 |Df(x+th)[h]|\leq L\norm h.
\]
The fundamental theorem of calculus therefore gives
\begin{equation}\label{eq:increment-D}
 |f(x+h)-f(x)|\leq L\norm h
\end{equation}
for $\Gamma$-almost every $x\in V_h$, hence for a dense subset of $V_h$. By continuity of $f$ we deduce it holds for every $x\in V_h$.

Now let $x,y\in\Omega$.  Density of $H_\Gamma$ provides a sequence $(h_j)$ in $H_\Gamma$
with $h_j\to y-x$.  For all sufficiently large $j$, $x+h_j\in\Omega$, so $x\in V_{h_j}$ and \eqref{eq:increment-D} gives $|f(x+h_j)-f(x)|\leq L\|h_j\|$. Passing to the limit gives $|f(y)-f(x)|\leq L\norm{y-x}$. Thus $f$ is $L$-Lipschitz, as required.

\smallskip

{\bf Step 2.} Let $D$ be the operator from Step 1 and let $\Phi$ be the operator from Lemma~\ref{lem:field-functional} applied to measure $\gamma$. Then $\Phi\circ D$ is an isometric embedding of $\Lipz(\Omega)$ into $(L^1(\gamma;X))^*$. Let $S=(\Phi\circ D)^*|_{L^1(\gamma,X)}$. Then $S:L^1(\gamma;X)\to (\Lipz(\Omega))^*$ is given by the formula
$$S(\g)(f)=\int_\Omega \pair{f'(x)}{\g(x)}\,d\gamma(x),\quad f\in\Lipz(\Omega),\g\in L^1(\gamma;X).$$

\smallskip

Indeed, $\Phi\circ D$ is an isometric embedding, being the composition of two isometric embeddings. Further, by the definitions we have
$$S(\g)(f)=\kappa(\g)(\Phi(D f))= \Phi(f')(\g),$$
where $\kappa$ denotes the canonical embedding of $L^1(\gamma;X)$ into its bidual. Thus the formula follows from the formula from Lemma~\ref{lem:field-functional}.

\smallskip

{\bf Step 3.} If $\g=\psi h$, where $h\in H_\Gamma$ and $\psi:\Omega\to \er$ is a 
Lipschitz function with bounded support such that $\dist(\supp\psi,X\setminus\Omega)>0$,
then $S(\g)\in\F(\Omega)$. More precisely,
$$S(\g)=\int_\Omega \delta_\Omega(x)(b_h(x)\psi(x)-\partial_h\psi(x))\, d\gamma(x),$$
where the integral is in the Bochner sense and $b_h$ is the function provided by property \ref{GIP}.

\smallskip

Let us first observe that the Bochner integral exists. The integrand is strongly measurable as 
the mapping $\delta_\Omega$ is continuous, $b_h\in L^1(\Gamma)$, hence it is measurable, $\psi$ is continuous, $\partial_h\psi$ is Borel measurable and, additionally, $\F(\Omega)$ is separable. Let us show that it is also Bochner integrable. Note that $\partial_h\psi$ is bounded, $\psi$ is also bounded, being a Lipschitz function with bounded support, $b_h\in L^1(\Gamma)$ and hence $b_h|_\Omega\in L^1(\gamma)$. Finally, $\norm{\delta_\Omega(x)}=\norm{x}$ for each $x\in\Omega$, hence it is bounded on the support of $\psi$. By combining these facts we deduce the Bochner integrability. Hence the above Bochner integral defines some element of $\F(\Omega)$. It remains to show that this element coincides with $S(\g)$.

To this end fix any $f\in \Lipz(\Omega)$. This function admits a Lipschitz extension to $X$; denote it again by $f$. Further, extend $\psi$ by zero outside $\Omega$ and denote the extension again by $\psi$. Then $\psi$ is Lipschitz on $X$ (as $\dist(\supp\psi,X\setminus\Omega)>0$). Let us now compute
$$\begin{gathered}
\pair{f}{\int_\Omega \delta_\Omega(x)(b_h(x)\psi(x)-\partial_h\psi(x))\, d\gamma(x)}=
\int_\Omega f(x)(b_h(x)\psi(x)-\partial_h\psi(x))\, d\gamma(x)
\\=\frac1{\Gamma(\Omega)}\int_X f(x)(b_h(x)\psi(x)-\partial_h\psi(x))\, d\Gamma(x)
=\frac1{\Gamma(\Omega)}\int_X \partial_h f(x)\psi(x)\, d\Gamma(x) \\
=\int_\Omega  \partial_h f(x)\psi(x)\, d\gamma(x) =\int_\Omega \pair{f'(x)}{h}\psi(x)\, d\gamma(x)
=\int_\Omega \pair{f'(x)}{\psi(x)h}\, d\gamma(x)=S(\g)(f),
\end{gathered}$$
where the first equality follows from properties of the Bochner integral (commuting with bounded linear operators), and the third equality follows from property \ref{GIP}. This completes the argument.

\smallskip

{\bf Step 4.} The functions of the form from Step 3 are linearly dense 
in $L^1(\gamma;X)$.

\smallskip

Let us start by showing that the respective scalar functions $\psi$ are dense in $L^1(\gamma)$.
Since the described class of Lipschitz functions is a vector space and simple measurable functions are dense in $L^1(\gamma)$, it is enough to approximate characteristic functions.  Let $A\subset\Omega$ be Borel and let $\varepsilon>0$.
Radon regularity gives a compact set $K\subset A$ such that $\gamma(A\setminus K)<\varepsilon$. Let $V\supset K$ be open and bounded such that $\gamma(V\setminus K)<\varepsilon$ and $\dist(\overline{V},X\setminus\Omega)>0$. Then \[
 \psi(x)=
 \frac{\dist(x,X\setminus V)}
      {\dist(x,K)+\dist(x,X\setminus V)},
 \qquad x\in \Omega
\]
is  Lipschitz, equals $1$ on $K$, and vanishes on
$X\setminus V$. Hence $\supp\psi\subset\overline V$, so $\dist(\supp\psi,X\setminus\Omega)>0$.
 Moreover, 
\[
 \norm{\psi-\mathbf1_A}_{L^1(\gamma)}
 \leq\gamma(V\setminus A)+\gamma(A\setminus K)\leq \gamma(V\setminus K)+\varepsilon<2\varepsilon,
\]
which completes the argument.

Taking into account that $H_\Gamma$ is dense in $X$, the functions from Step 3 approximate any function of the form $f\cdot x$, where $f\in L^1(\gamma)$ and $x\in X$. It is well known and easy to check that these functions are linearly dense in $L^1(\gamma;X)$ (see, e.g., \cite[Example 2.19]{ryan}, it is enough to observe that simple functions are linear combinations of functions $f\cdot x$). This completes the proof.

\smallskip

{\bf Step 5.} Conclusion: $S$ is an isometric quotient of $L^1(\gamma;X)$ onto $\F(\Omega)$ and $S^*=\Phi\circ D$.

\smallskip

By Steps 3 and 4 $S$ maps a dense subspace of $L^1(\gamma;X)$ into $\F(\Omega)$. Since $S$ is continuous, it maps $L^1(\gamma;X)$ into $\F(\Omega)$. Now it follows by the very definition of $S$ in Step 2 that $S^*=\Phi\circ D$. Since $S^*$ is an isometric embedding (by Step 2 and the preceding sentence), $S$ is an isometric quotient by Proposition~\ref{prop:subspaces-quotients}$(iii)$. This completes the proof of the theorem (the formula for $S$ follows from Step 2).

\section{Two applications of the representation theorem}
\label{sec:first-consequences}

The isometric quotient representation has two immediate classical
consequences which help motivate its construction.  First, it recovers the theorem of Godefroy and Kalton that every separable Banach space has the
isometric lifting property; see \cite[Theorem~3.1]{GodefroyKalton}. A closely
related formula already appears in their alternative proof based on Gaussian averaging. Here, however, the lifting arises naturally and directly from the representation operator.
Second, on passing to adjoints, the same representation recovers, in the separable case, Lindenstrauss's theorem that $X^*$ is $1$-complemented in $\Lipz(X)$; see \cite{Lindenstrauss} and also \cite[Chapter~7]{BenyaminiLindenstrauss}. Moreover, the representation simultaneously yields localized versions of both conclusions for every open convex subset $\Omega\subset X$ containing $0$.

Throughout this section, let $X$, $\Omega$, $\gamma$, and $S:L^1(\gamma;X)\longrightarrow\F(\Omega)$ be as in Theorem~\ref{thm:main}. Let $\beta_\Omega:\F(\Omega)\longrightarrow X$ be the linearization of the inclusion $\Omega\hookrightarrow X$. This is the usual barycenter map, see e.g. \cite[p. 124]{GodefroyKalton}. We say that a Banach space $X$ has the \emph{isometric lifting property} if there exists a norm-one linear operator $T:X\to \F(X)$ satisfying $\beta_X T = \operatorname{Id}_X$. 

\begin{proposition}\label{prop:barycenter-and-lifting}
Let $X$ be a nontrivial separable Banach space, let
$\Omega\subset X$ be a nonempty open convex set containing $0$, and let
$\gamma$ and $S:L^1(\gamma;X)\to\F(\Omega)$ be as in
Theorem~\ref{thm:main}.
Then
\begin{equation}\label{eq:barycenter-compatibility}
    \beta_\Omega(S \g)=\int_\Omega \g\, d\gamma,\quad \g\in L^1(\gamma;X).
\end{equation}
Moreover, the operator $R_\gamma:X\to\F(\Omega)$ defined by
$$R_\gamma(x)=S(x), \quad x\in X,$$
where $x$ plugged in $S$ means the constant mapping equal to $x$ on $\Omega$,
is a linear isometric embedding satisfying 
\begin{equation}\label{eq:localized-lifting-right-inverse}
 \beta_\Omega R_\gamma=\operatorname{Id}_X.
\end{equation}
In particular, $X$ has the isometric lifting property.
\end{proposition}

\begin{proof}
For $\g\in L^1(\gamma;X)$ and $x^*\in X^*$,
Fact~\ref{fact:linearization-adjoint}, applied to the inclusion
$\Omega\hookrightarrow X$, and the formula for $S$ in
Theorem~\ref{thm:main} give
\[
 \pair{\beta_\Omega S\g}{x^*}
 =\pair{S\g}{x^*|_\Omega}
 =\int_\Omega x^*(\g(t))\,d\gamma(t)=\pair{\int_\Omega\g\,d\gamma}{x^*},
\]
since $(x^*|_\Omega)'(t)=x^*$ for every $t\in\Omega$.
As $x^*\in X^*$ was arbitrary, this proves
\eqref{eq:barycenter-compatibility}.

Applying this identity to the constant function equal to $x\in X$
and using $\gamma(\Omega)=1$, we obtain
$\beta_\Omega R_\gamma x=x$.
Moreover, $R_\gamma$ is linear, and the contractivity of $S$
and $\beta_\Omega$ yields
\[
 \norm{x}
 =\norm{\beta_\Omega R_\gamma x}
 \le\norm{R_\gamma x}
 \le\norm{x}_{L^1(\gamma;X)}
 =\norm{x}.
\]
Thus $R_\gamma$ is an isometric embedding satisfying
\eqref{eq:localized-lifting-right-inverse}.
Taking $\Omega=X$ gives the isometric lifting property.
\end{proof}

Note that the above result does more than give a different proof of
\cite[Theorem~3.1]{GodefroyKalton}: it identifies both maps occurring in
the lifting identity within the isometric quotient representation.  The
barycenter map is induced by taking the Bochner mean of a vector field,
whereas its norm-one right inverse is obtained by passing the constant
vector fields through the quotient map $S$. 

Using dual operators we get the following corollary.

\begin{corollary}\label{cor:averaged-derivative-projection}
In the setting of \Cref{prop:barycenter-and-lifting} the dual mapping $R_\gamma^*$ is a norm-one projection of $\Lipz(\Omega)=\F(\Omega)^*$ onto $X^*$. Further, it is given by the formula
\[
 (R_\gamma^* f)(x)
 =\int_\Omega \partial_xf\,d\gamma  \quad f\in\Lipz(\Omega),\ x\in X.
\]
Hence, the canonical copy of $X^*$ in $\Lipz(\Omega)$ is
$1$-complemented.
\end{corollary}

\begin{proof}
Let us compute the formula for $R_\gamma^*$.
For $f\in\Lipz(\Omega)$ and $x\in X$ we have
$$R_\gamma^*f (x)=\pair{f}{R_\gamma x}=\pair{f}{Sx}=\int_\Omega f'(t)(x) \, d\gamma(t)
=\int_\Omega\partial_x f\, d\gamma,$$
which is exactly the promised formula.

Moreover, if $f=x^*\in X^*\subset \Lipz(\Omega)$ (more precisely, $f=x^*|_\Omega$), then
$$R_\gamma^*x^*(x)=\int_\Omega \partial_x x^*\, d\gamma=\int_\Omega x^*(x)\, d\gamma=x^*(x),$$
thus $R_\gamma^*x^*=x^*$. It follows that it is a projection. It has norm one as $R_\gamma$ is an isometric embedding.
\end{proof}

For $\Omega=X$, the previous corollary gives the separable case of Lindenstrauss's
theorem on complementability of $X^*$ in $\Lipz(X)$. Additionally, it provides an explicit formula for this projection: it assigns to a Lipschitz function its averaged derivative.

\begin{remark}
    We note that the embedding in Proposition~\ref{prop:barycenter-and-lifting} and the projection in Corollary~\ref{cor:averaged-derivative-projection} depend on the choice of the measure $\gamma$. Moreover, even for fixed $\gamma$ we may obtain a rich variety of embeddings and projections. Indeed, let $w\in L^1(\gamma)$ be any positive function with norm one. Then
    $$R_{\gamma,w}(x)=S(w\cdot x), \quad x\in X$$
    defines an isometric embedding of $X$ into $\F(\Omega)$. The dual mapping is then given by
    $$(R_{\gamma,w}^*f)(x)=\int_{\Omega} \partial_x f\cdot w \,d\gamma,\quad f\in\Lipz(\Omega), x\in X,$$
    and it is also a norm-one projection of $\Lipz(\Omega)$ onto $X^*$. The embedding $R_{\gamma,w}$ for a suitable weight $w$ will be used also in Section~\ref{sec:Mrep-complemented-implies-Hilbert} below. 
\end{remark}

\section{Norming subspace of smooth functions}\label{sec:norming-smooth}

In finite-dimensional spaces an important role is played by approximation by $C^\infty$-functions.
For example, in \cite{ckk-mathematika,CKK} such approximations (based mainly on convolutions with an approximate unit recalled in Section~\ref{subsec:convolution}) were strongly used. In this section we present a kind of substitute for this tool for infinite-dimensional Banach spaces. Some statements will be given for general Banach spaces as we are convinced they may be useful also elsewhere; some of them will be designed for our specific situation.

Let $X$ be a Banach space and let $W\subset X^*$ be a linear subspace.
A \emph{smooth $W$-cylindrical function} on $X$ is a function of the
form
\begin{equation}\label{eq:general-cylinder}
 \varphi(x)
 =g\bigl(w_1(x),\ldots,w_m(x)\bigr), x\in X,
 \mbox{ where }m\in\N,\  
 w_1,\ldots,w_m\in W,
 \ g\in C^\infty(\R^m).
\end{equation}
We denote the class of all such functions by
$\operatorname{Cyl}_W(X)$. The next lemma collects basic properties of such functions.

\begin{lemma}\label{L:cyl-smooth}
    Let $X$ be a Banach space, let $W\subset X^*$ be a linear subspace and let $\varphi\in\Cyl_W(X)$ be given  by \eqref{eq:general-cylinder}. Then the following assertions are valid:
    \begin{enumerate}[$(i)$]
        \item\label{it:smooth}  $\varphi\in C^\infty(X)$.
        \item\label{it:chainRule} The Fr\'echet derivative of $\varphi$ is given by
        $$\varphi'(x)=\sum_{j=1}^m \partial_j g\bigl(w_1(x),\ldots,w_m(x)\bigr)w_j,\quad x\in X.$$
        In particular, $\varphi'$ is a $\sigma(X,W)$-to-norm continuous mapping from $X$ into $W$.
        \item\label{it:LipOnBdd} $\varphi$ is Lipschitz on bounded sets.
        \item\label{it:convexFormulaBdd} If $K\subset X$ is a bounded convex set with nonempty norm-interior, then
        $$\norm{\varphi|_K}_{\mathrm{Lip}}=\norm{\varphi'|_K}_\infty=\norm{(\varphi|_{\Int K})'}_\infty.$$
    \end{enumerate}
\end{lemma}

\begin{proof}
    Assertion \ref{it:smooth} and the formula from \ref{it:chainRule} follow by the chain rule, see e.g. \cite[Theorem~1.83]{HajekJohanis}. The `in particular' part of \ref{it:chainRule} follows immediately from the formula (note that $\sigma(X,W)$ is, as usual, the weakest topology on $X$ making elements of $W$ continuous). 

    \ref{it:LipOnBdd}:  Let $M\subset X$ be bounded. Since $\co M$ is also bounded, we may assume without loss of generality that $M$ is convex. Set
\[
 Px=(w_1(x),\ldots,w_m(x)),\quad x\in X.
\]
Then $PM$ is a bounded convex subset of $\R^m$, hence $\overline{PM}$ is compact. Thus the (continuous) partial derivatives of $g$ are bounded on $\overline{PM}$. By the formula from \ref{it:chainRule} we deduce that $\varphi'$ is bounded on $M$. Hence $\varphi$ is Lipschitz on $M$.

\ref{it:convexFormulaBdd}: Since $\Int K$ is norm-dense in $K$ and $\varphi'$ is continuous on $X$, we get
$$\norm{\varphi|_K}_{\mathrm{Lip}}=\norm{\varphi|_{\Int K}}_{\mathrm{Lip}}=\norm{(\varphi|_{\Int K})'}_\infty=\norm{\varphi'|_{\Int K}}_\infty=\norm{\varphi'|_K}_\infty.$$
This completes the proof.
\end{proof}

Let us now pass to a localized version.
If $M\subset X$ is a  bounded subset containing $0$, put
\begin{equation}\label{eq:general-cylindrical-classes}
 \mathscr \Cyl_W(M)
 =\{\varphi|_M: \varphi\in\operatorname{Cyl}_W(X)\},
 \qquad
 \Cyl_W^0(M)=\{\varphi\in\Cyl_W(M):\varphi(0)=0\}.
\end{equation}
The next lemma records the elementary properties of these classes that
will be used below. 

\begin{lemma}
\label{lem:general-cylindrical-properties}
The following assertions hold.
\begin{enumerate}[$(a)$]
\item\label{item:G1}
The class $\Cyl_W(M)$ is a unital subalgebra of the algebra
of all real-valued functions on $M$. Moreover, $\Cyl_W^0(M)$ is an ideal in $\Cyl_W(M)$, i.e., it is a linear subspace and
$\Cyl_W(M)\Cyl_W^0(M)\subset\Cyl_W^0(M)$.

\item\label{item:G2}
Every function in $\Cyl_W(M)$ is Lipschitz on $M$.  In
particular,
\[
 \mathscr \Cyl_W^0(M)\subset\Lipz(M).
\]
\item\label{item:G3} If $W$ separates points of $M$,
then $\Cyl_W^0(M)$ separates points of $M$. If $W$ is $1$-norming for $X$, then $\Cyl_W^0(M)$ is $1$-norming for $\F(M)$.
\end{enumerate}
\end{lemma}

\begin{proof} 
\ref{item:G1}: It is clear that $\Cyl_W(M)$ is a unital subalgebra. The properties of $\Cyl_W^0(M)$ are also obvious.

\ref{item:G2}: This follows from Lemma~\ref{L:cyl-smooth}\ref{it:LipOnBdd}.

\ref{item:G3}: Assume that $W$ separates points of $M$. Since $\Cyl_W^0(M)$ contains $W$ (more precisely, the restrictions of elements of $W$ to $M$), it separates points of $M$ as well.

Next assume that $W$ is $1$-norming. To prove that $\Cyl_W^0(M)$ is $1$-norming, it is enough to show that every
$f\in B_{\Lipz(M)}$ belongs to the weak$^*$ closure of
$\Cyl_W^0(M)\cap B_{\Lipz(M)}$.  Fix such an $f$ and put $L=\norm f_{\mathrm{Lip}}\leq1$.
If $L=0$, then $f=0\in\Cyl_W^0(M)$, so there is nothing to prove.
So, assume $L>0$. Fix $R>0$ such that $M\subset RB_X$.

Note that the weak$^*$-topology on $B_{\Lipz(M)}$ coincides with the topology of pointwise convergence on $M$. So, fix $z_0,\dots,z_n\in M$ pairwise distinct and $\varepsilon>0$. Without loss of generality $z_0=0$ and $\varepsilon<1$. 
Since $W$ is $1$-norming, for each pair $0\leq i<j\leq n$ we may choose
$w_{ij}\in B_W$ such that
\[
 |w_{ij}(z_i-z_j)| 
 >(1-\varepsilon)\norm{z_i-z_j}.
\]
Let
\[
 I=\{(i,j):0\leq i<j\leq n\},
 \qquad
 m=|I|=\binom{n+1}{2},
\]
and define
\[
 T:X\longrightarrow\ell_\infty(I)
 \cong\ell_\infty^{m},
 \qquad
 Tx=\bigl(w_{ij}(x)\bigr)_{(i,j)\in I}\,,\qquad x\in X.
\]
Then $\|T\|\le 1$.  Moreover, for $0\leq i<j\leq n$,
\begin{equation}\label{eq:finite-almost-isometry}
 \norm{Tz_i-Tz_j}_\infty
 \geq
 |w_{ij}(z_i-z_j)|
 >
 (1-\varepsilon)\norm{z_i-z_j}.
\end{equation}
Thus $T$ is one-to-one on $\{z_0,\dots,z_n\}$ and, moreover, its inverse is $\frac1{1-\varepsilon}$-Lipschitz on $\{T(z_0),\dots,T(z_n)\}$. Thus $f\circ T^{-1}$ is $\frac{L}{1-\varepsilon}$-Lipschitz on $\{T(z_0),\dots,T(z_n)\}$. Let $F:\ell^\infty(I)\to\er$ be its $\frac{L}{1-\varepsilon}$-Lipschitz extension. Further, let $(u_k)$ be an approximate unit  in $\R^I=\R^m$ 
(see Section~\ref{subsec:convolution}). By Fact~\ref{fact:convolution-approximation} we know that
$u_k*F$ are $\frac{L}{1-\varepsilon}$-Lipschitz $C^\infty$ functions converging to $F$ uniformly on $\R^m$. Fix $k\in \N$ such that $\|u_k*F-F\|_\infty<\varepsilon$.
Define
\[
 \varphi(x)
 =
 (1-\varepsilon)
 \bigl((u_k*F)(Tx)-(u_k*F)(0)\bigr),
 \qquad x\in X.
\]
The function $\varphi$ belongs to $\operatorname{Cyl}_W(X)$ and satisfies $\varphi(0)=0$. Therefore $\varphi|_M\in\Cyl_W^0(M)$.
 Furthermore, it is $L$-Lipschitz, hence
\[
 \varphi|_M\in\Cyl_W^0(M)\cap B_{\Lipz(M)}.
\]
Finally, using that $F(Tz_i)=f(z_i)$ and therefore also $F(T0)=f(0)=0$, we obtain
\begin{align*}
 |\varphi(z_i)-f(z_i)|
 &\leq
\varepsilon|f(z_i)|
 +(1-\varepsilon)|(u_k*F)(Tz_i)-(u_k*F)(0)-f(z_i)|
 \\ &\leq
 \varepsilon L R 
 +(1-\varepsilon)
 \bigl(
   |(u_k*F)(Tz_i)-F(T z_i)|
   +|(u_k*F)(T0)-F(T0)|
 \bigr)\\
 &\leq
 \varepsilon R+ (1-\varepsilon)\cdot2\varepsilon.
\end{align*}
Since the right-hand side may be arbitrarily small, this completes the proof.
\end{proof}

We finally pass to the situation specific to our main results. If $X=Y^*$ for some nontrivial Banach space $Y$, then $W=Y$ (considered as a subspace of $Y^{**}=X^*$) is a possible choice. 
Note that then $Y$ is $1$-norming for $X$.

\begin{lemma}
\label{lem:dual-cylindrical-calculus}
Let $Y$ be a nontrivial Banach space, let $X=Y^*$ and let $K\subset X$ be a weak$^*$ compact convex set whose norm-interior contains $0$. 
The following assertions hold.

\begin{enumerate}[$(a)$]
\item\label{item:C1} If $\varphi\in \Cyl_Y(X)$, then $\varphi'$ is a weak$^*$-to-norm continuous mapping from $X$ to $Y$.

\item\label{item:C2} $\Cyl_Y(K)$ is a norm-dense unital subalgebra in $C(K)$.

\item\label{item:C3}
Every $\varphi\in\Cyl_Y(K)$ is Lipschitz on $K$. Moreover, 
\begin{equation}\label{eq:dual-cyl-norm}
 \norm{\varphi}_{\mathrm{Lip}} = \norm{(\varphi|_{\Int K})'}_\infty=\norm{\psi'|_K}_\infty,
\end{equation}
where $\psi\in\Cyl_Y(X)$ is such that $\varphi=\psi|_K$.

\item\label{item:C3.5} $\Cyl_Y^0(K)$ is a subspace of $\Lipz(K)$ $1$-norming for $\F(K)$.

\item\label{item:C4}
For $g\in\Cyl_Y(K)$ and $\varphi\in\Cyl_Y^0(K)$ we have,
\begin{equation}\label{eq:dual-product-rule}
 (g\varphi)'(x)=g(x)\varphi'(x)+g'(x)\varphi(x),\quad x\in \Int K,
\end{equation}
and consequently
\begin{equation}\label{eq:dual-product-bound}
 \norm{g\varphi}_{\mathrm{Lip}}
 \leq
 \bigl(\norm{g}_\infty+\rho\norm{(g|_{\Int K})'}_\infty\bigr)
 \norm{\varphi}_{\mathrm{Lip}},\mbox{ where }\rho=\sup_{x\in K}\norm{x}.
\end{equation}
\item\label{item:C6} If $\mu$ is a finite regular Borel measure on $K$, then $\Cyl_Y(K)$ is dense in $L^1(\mu)$.
\end{enumerate}
\end{lemma}

\begin{proof}
Assertion \ref{item:C1} follows from Lemma~\ref{L:cyl-smooth}\ref{it:chainRule}.

\ref{item:C2}: By the very definition of $\Cyl_Y(X)$ its elements are weak$^*$-continuous, hence $\Cyl_Y(K)\subset C(K)$. As stated in Lemma~\ref{lem:general-cylindrical-properties}, it is a unital subalgebra and, moreover, it separates points of $K$ (as $Y$ separates points of $X=Y^*$). We thus conclude by the Stone-Weierstrass theorem.

\ref{item:C3}: This follows from Lemma~\ref{L:cyl-smooth}\ref{it:convexFormulaBdd}.

\ref{item:C3.5}: This follows from Lemma~\ref{lem:general-cylindrical-properties}\ref{item:G3}.

\ref{item:C4}: Formula \eqref{eq:dual-product-rule} follows from the ordinary product rule for derivatives, see e.g. \cite[Corollary~1.116]{HajekJohanis}.
 Since $\varphi(0)=0$, we get
$\norm\varphi_\infty\leq\rho\norm{\varphi}_{\mathrm{Lip}}$.
Combining this with \eqref{eq:dual-cyl-norm} yields
\eqref{eq:dual-product-bound}.

\ref{item:C6} $C(K)$ is dense in $L^1(\mu)$ by the Luzin theorem. By \ref{item:C2} $\Cyl_Y(K)$ is uniformly dense in $C(K)$, so it is dense in $L^1(\mu)$ (as $C(K)$ embeds continuously into $L^1(\mu)$). 
 \end{proof}

\section{Complementability for bounded domains}\label{sec:bounded}

In this section we prove Theorem~\ref{thm:complement} in the case where $\Omega$ is additionally bounded. The line of this proof is similar to that in \cite{CKK}. Although the result holds only for Hilbert spaces, several steps work in a more general setting. 

The basic setting in this section is the following. Let $Y$ be a nontrivial Banach space and let $X=Y^*$. Assume that $X$ is separable (then $Y$ is necessarily separable as well). Assume that $K\subset X$ is a weak$^*$ compact convex set with $0$ in its norm-interior. Set $\Omega=\Int K$. Further, let $\Gamma$ be a non-degenerate Gaussian measure on $X$ and let $\gamma$ be the normalized restriction of $\Gamma$ to $\Omega$ (equivalently, to $K$, cf. 
Remark~\ref{rem:gauss}(2)). Further, let $S:L^1(\gamma;X)\to\F(\Omega)$ be the quotient operator provided by Theorem~\ref{thm:main}.

Next observe that the canonical inclusion $C(K,Y)\subset L^\infty_{w^*}(\gamma;X^*)$ is an isometric embedding (we consider $Y$ canonically embedded into $Y^{**}=X^*$). Hence, if $\Phi$ is the embedding from Lemma~\ref{lem:field-functional} (applied to the measure $\gamma$), then $\Phi|_{C(K,Y)}$ is an isometric embedding of $C(K,Y)$ into $(L^1(\gamma;X))^*$. We take the dual operator
$$R=(\Phi|_{C(K,Y)})^*:(L^1(\gamma;X))^{**}\to C(K,Y)^*=M(K,X).$$
By Proposition~\ref{prop:subspaces-quotients} $R$ is an isometric quotient. (Note that this operator is an analogue of the operator $\boldsymbol{R}$ from \cite{CKK}.)

The following lemma shows, in particular, a connection between $R$ and $S$.

\begin{lemma}\label{lem:dual-R-construction} The operator $R$ satisfies the following properties:
\begin{enumerate}[$(a)$]
    \item\label{it:Rformula} If $\g\in L^1(\gamma;X)$ is considered as an element of the bidual, then
    $R(\g)=\g\gamma$.
    \item\label{it:R**formula} $\pair{R\g^{**}}{\varphi'}=\pair{S^{**}\g^{**}}{\varphi}$ whenever $\g^{**}\in (L^1(\gamma;X))^{**}$ and $\varphi\in\mathscr \Cyl_Y^0(K)$.
\end{enumerate}
\end{lemma}

\begin{proof}
\ref{it:Rformula}: Let $\f\in C(K,Y)$. Let us compute:
\[
 \pair{R\g}{\f} = \pair{\g}{\Phi(\f)}=\int_\Omega\pair{\g(x)}{\f(x)}\,d\gamma(x)=\int_\Omega \f\,d(\g\gamma)=\pair{\g\gamma}{\f},
\]
where the third equality follows easily from \eqref{eq:dual-vector-pairing}. This completes the argument.

\ref{it:R**formula}: Let $\g^{**}\in (L^1(\gamma;X))^{**}$ and $\varphi\in\Cyl_Y^0(K)$. For any $\g\in L^1(\gamma;X)$ we have
\[
 \pair{S^*\varphi}{\g}
 =\int_\Omega\pair{\g(x)}{\varphi'(x)}\,d\gamma(x)
 =\int_K\pair{\g(x)}{\varphi'(x)}\,d\gamma(x)
   =\pair{\Phi(\varphi')}{\g}.
\]
Thus $S^*\varphi=\Phi(\varphi')$ in $(L^1(\gamma;X))^*$. Consequently, the definition of $R$ yields
\[
 \pair{R\g^{**}}{\varphi'}
 =\pair{\g^{**}}{\Phi(\varphi')}
 =\pair{\g^{**}}{S^*\varphi}
 =\pair{S^{**}\g^{**}}{\varphi},
\]
which completes the proof.
\end{proof}

 It follows from Lemma~\ref{lem:dual-cylindrical-calculus}\ref{item:C3} that the assignment $\varphi\mapsto \varphi'$ is an isometric embedding of $\Cyl_Y^0(K)$ (considered as a subspace of $\Lipz(K)$) into $C(K,Y)$. Thus the dual mapping which we will denote by $\theta$ is an isometric quotient
map of $M(K,X)=C(K,Y)^*$ onto $(\Cyl_Y^0(K))^*$ (by Proposition~\ref{prop:subspaces-quotients}). It is given by
$$\theta_{\bmu}(\varphi)=\pair{\bmu}{\varphi'},\quad \varphi\in\Cyl_Y^0(K), \bmu\in M(K,X).$$
Let us call $\mu\in M(K,X)$ representable if $\theta_\mu$ is represented by an element of $\F(K)$, i.e., if $\theta_\mu=\kappa(x)|_{\Cyl_Y^0(K)}$ for some $x\in \F(K)$. Here $\kappa$ denotes the canonical embedding of $\F(K)$ into its bidual. Since $\Cyl_Y^0(K)$ is $1$-norming (by Lemma~\ref{lem:general-cylindrical-properties}\ref{item:G3}), the assignment
$x\mapsto \kappa(x)|_{\Cyl_Y^0(K)}$ is an isometric embedding, in particular its range is closed. It follows that the space of representable measures is closed and $\theta$ restricted to it may be interpreted as an isometric quotient of representable measures onto $\F(K)$. Denote the space of representable measures by $M_{\mathrm{rep}}(K;X)$ and the respective quotient map by $q$. 

\begin{lemma}\label{lem:fgammarepres}
   For each $\f\in L^1(\gamma,X)$ the measure $\f\gamma$ is representable and $q(\f\gamma)=S(\f)$.
\end{lemma}

\begin{proof}
 Fix $\varphi\in\Cyl_Y^0(K)$ and compute using definitions:
 $$\begin{aligned}
   \theta_{\f\gamma}(\varphi)&=\pair{\f\gamma}{\varphi'}=\int_K\pair{\f(x)}{\varphi'(x)}\,d\gamma(x)=\pair{\varphi}{S(\f)}.  
 \end{aligned}$$
\end{proof}

\begin{proposition}\label{prop:dual-module}
If $\boldsymbol\mu\in M_{\mathrm{rep}}(K;X)$ and
$g\in L^1(|\boldsymbol\mu|)$, then
$g\boldsymbol\mu\in M_{\mathrm{rep}}(K;X)$.
Moreover, if $g\in\Cyl_Y(K)$ and
$\boldsymbol\mu=\uu|\boldsymbol\mu|$ with $\norm{\uu(x)}=1$
for $|\boldsymbol\mu|$-almost every $x\in K$, then
\begin{equation}\label{eq:dual-module-representative}
 q(g\boldsymbol\mu)
 =W_g(q\boldsymbol\mu)
  -\int_K\delta_K(x)\pair{\uu(x)}{g'(x)}\,d|\boldsymbol\mu|(x),
\end{equation}
where $W_g$ is the operator from Lemma~\ref{lem:free-weighting}.
\end{proposition}

\begin{proof}
Fix $\boldsymbol\mu\in\mathcal M_{\mathrm{rep}}(K;X)$ and write
$\boldsymbol\mu=\uu|\boldsymbol\mu|$ as in Section~\ref{ssec:vectormeasures} (recall that $X$ is a separable dual space). Fix $\rho>0$ such that $K\subset B(0,\rho)$.
First suppose that $g\in\Cyl_Y(K)$. Observe that the map $\delta_K$ is strongly
measurable for $|\boldsymbol\mu|$ (it is norm-to-norm continuous, thus weak$^*$-to norm Borel due to separability of $X=Y^*$ and hence strongly measurable due to separability of $\F(K)$). Since $\uu$ is strongly measurable as well
and $g'\in C(K;Y)$, the scalar function
$x\mapsto\pair{\uu(x)}{g'(x)}$ is measurable. Moreover,
\[
 \int_K\norm{\delta_K(x)}\,|\pair{\uu(x)}{g'(x)}|
       \,d|\boldsymbol\mu|(x)
 \leq\rho\norm{g'}_\infty\norm{\boldsymbol\mu}.
\]
Consequently, the Bochner integral in
\eqref{eq:dual-module-representative} exists in $\F(K)$.
Let $\xi$ denote the right-hand side of that formula.
For $\varphi\in\Cyl_Y^0(K)$, we have
$g\varphi\in\Cyl_Y^0(K)$ by Lemma~\ref{lem:general-cylindrical-properties}\ref{item:G1}.
Using $W_g^*\varphi=g\varphi$ (see Lemma~\ref{lem:free-weighting}) and the product rule in
Lemma~\ref{lem:dual-cylindrical-calculus}\ref{item:C4}, we obtain
\begin{align*}
 \pair{\xi}{\varphi}
 &=\pair{q\boldsymbol\mu}{g\varphi}
   -\int_K\varphi(x)\pair{\uu(x)}{g'(x)}\,d|\boldsymbol\mu|(x)\\
 &=\int_K\pair{\uu(x)}{(g\varphi)'(x)-\varphi(x)g'(x)}
           \,d|\boldsymbol\mu|(x)\\
 &=\int_K g(x)\pair{\uu(x)}{\varphi'(x)}\,d|\boldsymbol\mu|(x)
 =\theta_{g\boldsymbol\mu}(\varphi).
\end{align*}
Thus $g\boldsymbol\mu$ is represented by $\xi$, proving
\eqref{eq:dual-module-representative}.

The assignment $g\mapsto g\boldsymbol{\mu}$ is a bounded linear mapping from $L^1(\abs{\boldsymbol{\mu}})$ to $M(K,X)$. Hence the set of those $g$ for which $g\boldsymbol{\mu}$ is representable is closed. Since it contains $\Cyl_Y(K)$, it is also dense (by Lemma~\ref{lem:dual-cylindrical-calculus}\ref{item:C6}).
This completes the proof.
\end{proof}

For a finite positive Borel measure $\nu$ on $K$, put
\begin{equation}\label{eq:dual-Mnu}
 M_\nu
 =\{\g\in L^1(\nu;X):
       \g\nu\in M_{\mathrm{rep}}(K;X)\}.
\end{equation}

\begin{proposition}\label{prop:tangent-fields}
    Let $\nu$ be a finite positive Borel measure on $K$. Then there is a sequence $(\uu_n)$ of strongly measurable $X$-valued functions on $K$ such that for any $\g\in L^1(\nu;X)$ we have
    $$ \g\in M_\nu\iff \g(x)\in\overline{\linspan}\{\uu_n(x);n\in\N\}\mbox{ for $\nu$-almost all }x\in K.$$
\end{proposition}

\begin{proof}
    This proposition is an analogue of \cite[Proposition 7.1]{CKK} and the proof is similar. 
   We first observe that $M_\nu$ satisfies the assumptions of \cite[Theorem~3.1]{HiaiUmegaki} (for $p=1$) -- it is enough to use that $M_\nu$ is closed in $L^1(\nu;X)$ and Proposition~\ref{prop:dual-module}. Hence, by combining  \cite[Theorem~3.1]{HiaiUmegaki} with
   \cite[Lemma~1.1]{HiaiUmegaki}, there is a sequence
$(\uu_j)\subset M_\nu$ such that
 $$\forall \g\in L^1(\nu;X)\colon \g\in M_\nu\iff \g(x)\in\overline{\{\uu_n(x);n\in\N\}}\mbox{ for $\nu$-almost all }x\in K.$$
Using the fact that $M_\nu$ is a linear subspace, we may easily achieve that $\overline{\{\uu_n(x);n\in\N\}}$ is a linear subspace for $\nu$-almost all $x\in K$ (cf. the proof of \cite[Proposition 7.1]{CKK}). This completes the proof.
\end{proof}

Recall that we aim at proving that $\F(\Omega)$ is sometimes complemented in $\F(\Omega)^{**}$. The already proved results reduce this task to finding a projection of $M(K;X)$ onto $M_{\mathrm{rep}}(K;X)$. Indeed, assume that $P$ is such a projection. Then $qPR$ is a bounded linear surjection of $(L^1(\gamma;X))^{**}$ onto $\F(K)$. The required projection may be found by a suitable factorization of this operator. This may be formulated and proved in a precise way.
However, as the existence of $P$ requires Hilbert space structure (by Theorem~\ref{thm:Intro1}), we do not formulate the abstract version. We only prove the Hilbert space version in Proposition~\ref{prop:dual-complementability-transfer} below.

In the rest of this section we assume that $X$ (and hence also $Y$) is a separable Hilbert space with the inner product $\pair{\cdot}{\cdot}$. We also use the standard representation of the dual of a Hilbert space using the inner product. Hilbert space structure allows us to use orthogonal projections. In order to use it properly we need the following parametrized version of the Gram--Schmidt procedure.

\begin{lemma}\label{lem:G-S}
    Let $(\Omega,\Sigma,\nu)$ be a finite measure space, let $H$ be a Hilbert space. Assume that
    $(\uu_n)$ is a sequence of strongly $\nu$-measurable $H$-valued functions. Then there is another sequence $(\e_n)$ of strongly $\nu$-measurable $H$-valued functions with the following properties:
    \begin{itemize}
        \item[$(i)$] $\linspan\{\e_n(t);n\in\N\}=\linspan\{\uu_n(t);n\in\N\}$ for $\nu$-almost all $t\in\Omega$;
        \item[$(ii)$] $\norm{\e_n(t)}\in\{0,1\}$ for each $n\in\N$ and $\nu$-almost all $t\in\Omega$;
        \item[$(iii)$] $\pair{\e_n(t)}{\e_m(t)}=0$ whenever $n\ne m$ for  $\nu$-almost all $t\in\Omega$.
    \end{itemize}
\end{lemma}

\begin{proof}
    We apply a slight modification of the Gram--Schmidt procedure to the sequence $(\uu_n(t))$ for each fixed
$t$. More precisely, define recursively
\[
 \vv_n(t)=\uu_n(t)-\sum_{j=1}^{n-1}
                   \pair{\uu_n(t)}{\e_j(t)}\e_j(t),
 \qquad
 \e_n(t)=
 \begin{cases}
  \vv_n(t)/\norm{\vv_n(t)},&\vv_n(t)\ne0,\\
  0,&\vv_n(t)=0,
 \end{cases}
\]
where the sum is empty when $n=1$. Note that, unlike in the standard Gram--Schmidt procedure, we keep the zero vectors in the sequence, so that every $\e_n$
is defined on all of $\Omega$, even though the linear dependence
relations among the $\uu_j(t)$ may vary with $t$.
\end{proof}

\begin{lemma}
\label{lem:fibrewise-orthogonal-projections}
Assume that $X$ (and hence also $Y$) is a separable Hilbert space and $\nu$ is a finite positive Borel measure on $K$. Let $(\uu_n)$ be a sequence provided by Proposition~\ref{prop:tangent-fields}.
For $x\in K$ denote by $\pi(x)$ the orthogonal projection of $X$ onto $\overline{\linspan}\{\uu_n(x);n\in\N\}$. Then the map
\[
 (P_\nu \f)(x)=\pi(x)\f(x), \quad x\in K,\f\in L^1(\nu;X)
\]
is a well-defined contractive linear projection of $L^1(\nu;X)$ onto $M_\nu$.
\end{lemma}

\begin{proof}
    Let $(\e_n)$ be the sequence provided by Lemma~\ref{lem:G-S}. Then we have, for each $\f\in L^1(\nu;X)$, 
    \[
 \pi(x)\f(x)=\sum_{j=1}^{\infty}
             \pair{\f(x)}{\e_j(x)}\e_j(x)
 \qquad x\in K,
\]
where the series converges in the norm of $X$. It easily follows that $x\mapsto \pi(x)\f(x)$ is a well-defined strongly measurable function. Since the orthogonal projection is contractive, it is also Bochner integrable.
Thus $P_\nu$ is a well-defined mapping of $L^1(\nu;X)$ into $L^1(\nu;X)$. It is clearly linear; by the above it is contractive. Moreover, it follows from the characterization from Proposition~\ref{prop:tangent-fields} that $P_\nu$ is a projection with range  $M_\nu$.  
\end{proof}

\begin{proposition}\label{prop:glue-Pnu}
Assume that $X$ is a separable Hilbert space. For $\bmu\in M(K,X)$ such that $\bmu=\g\nu$ for some $\nu\ge0$ and $\g\in L^1(\nu;X)$ set $$P \bmu=(P_\nu\g)\nu.$$
Then $P$ is a well-defined contractive linear projection of $M(K,X)$ onto $M_{\mathrm{rep}}(K;X)$.   
\end{proposition}

\begin{proof}
    The proof may be done by copying the proofs of Lemma 7.3 and Proposition 7.4 from \cite{CKK}, up to adapting a bit the notation: It is enough to replace $\mathbb R^d$ by $K$, $E$ by $X$, $M_\sigma(\mathbb R^d,E)$ by $M(K,X)$, $M_\sigma^{\mathrm{div}}(E)$ by $M_{\mathrm{rep}}(K;X)$, and $T_\nu(x)$ by $\overline{\linspan}\{\uu_n(x):n\in\N\}$, where $(\uu_n)$ is provided by Proposition~\ref{prop:tangent-fields}. Further,
    Proposition~\ref{prop:tangent-fields} should be applied in place of \cite[Proposition~7.1]{CKK} and Lemma~\ref{lem:fibrewise-orthogonal-projections} in place of \cite[Proposition~7.2]{CKK}.
\end{proof}

\begin{proposition}
\label{prop:dual-complementability-transfer}
Assume that $X$ is a separable Hilbert space.
 The formula
\begin{equation}\label{eq:dual-QP-factorization}
 Q(S^{**}\g^{**})=qPR\g^{**},\quad \g^{**}\in (L^1(\gamma;X))^{**}.
\end{equation}
defines a contractive linear projection $Q:\F(K)^{**}\to\F(K)$.
Moreover, if $\xi^{**}\in\F(K)^{**}$ and
$\boldsymbol\mu\in\mathcal M(K;X)$ satisfy $\theta_{\bmu}=\xi^{**}|_{\Cyl_Y^0(K)}$, then
\begin{equation}\label{eq:dual-QP-representative}
 Q\xi^{**}=qP\boldsymbol\mu.
\end{equation}
\end{proposition}

\begin{proof}
We start by noting that $S^{**}$ is an isometric quotient of $(L^1(\gamma;X))^{**}$ onto $\F(K)^{**}$ (by Proposition~\ref{prop:subspaces-quotients}), hence any element of $\F(K)^{**}$ is of the form $S^{**}\g^{**}$.

Set $Q_0=qPR$. Then $Q_0$ is an isometric quotient of $(L^1(\gamma;X))^{**}$ onto $\F(K)$ (being a composition of three isometric quotients). Next observe that $\ker S^{**}\subset\ker Q_0$. Indeed, let $\g^{**}\in\ker S^{**}$. By Lemma~\ref{lem:dual-R-construction}$(b)$ we get
$$\pair{R\g^{**}}{\varphi'}=0,\quad \varphi\in\Cyl_Y^0(K).$$
Thus $\theta_{R\g^{**}}=0$. It follows that $R\g^{**}$ is representable (hence $PR\g^{**}=R\g^{**}$) and $qR\g^{**}=0$. So, $Q_0\g^{**}=0$ and the inclusion $\ker S^{**}\subset \ker Q_0$ follows. Therefore, $Q$ is a well-defined isometric quotient. 

To prove that $Q$ is a projection, it remains to show that it fixes $\F(K)$. Let $\xi\in\F(K)$ and choose $\g\in L^1(\gamma;X)$ with $S\g=\xi$.
Identifying $L^1(\gamma;X)$ and $\F(K)$ with their canonical copies in their
biduals, we get
\[
 Q\xi
 =QS^{**}\g
 =qPR\g
 =qP(\g\gamma)
 =q(\g\gamma)
 =S\g=\xi.
\]
Here the third equality follows from Lemma~\ref{lem:dual-R-construction}$(a)$ and the following two equalities follow from Lemma~\ref{lem:fgammarepres}. This completes the argument.

Finally, let $\xi^{**}$ and $\boldsymbol\mu$ satisfy the additional
hypothesis. Choose any $\g^{**}\in (L^1(\gamma;X))^{**}$ with $S^{**}\g^{**}=\xi^{**}$. It follows from Lemma~\ref{lem:dual-R-construction}$(b)$ that $\theta_{\bmu}=\theta_{R\g^{**}}$, thus $\theta_{\bmu-R\g^{**}}=0$, so $\bmu-R\g^{**}$ is representable and $q(\bmu-R\g^{**})=0$. Thus $P(\bmu-R\g^{**})=\bmu-R\g^{**}$, so
$$Q\xi^{**}-qP\bmu=QS^{**}\g^{**}-qP\bmu=qP(R\g^{**}-\bmu)=q(R\g^{**}-\bmu)=0,$$
which completes the proof.
\end{proof}

\begin{remark}\label{rem:dim-bounded}
The previous proposition is formulated and proved for bounded open convex subsets of a Hilbert space. As noted in the introduction, the same result may be proved for bounded open convex subsets of a general Banach space of dimension at most $2$. If the dimension is at most one, then the space is isometric to a Hilbert space, hence the only new part is the statement on two-dimensional Banach spaces. So, let us indicate the necessary changes.

Assume that $X$ is a Banach space of dimension $2$ and $\Omega\subset X$ is a bounded open convex set containing $0$. Then Theorem~\ref{thm:main}, results from Section~\ref{sec:norming-smooth} and from the first part of Section~\ref{sec:bounded} (till Proposition~\ref{prop:tangent-fields}) work, as they are valid for general separable dual spaces. At this point we apply, instead of orthogonal projections, the ideas from \cite{BGLMPS26b}: 

Since $\dim X=2$, any subspace of $X$ is $1$-complemented. Let $\widetilde{P}$ be the Borel mapping provided by \cite[Proposition 4.4]{BGLMPS26b} which assigns to each subspace $F\subset X$ a norm-one projection $\widetilde{P}(F):X\to F$. Given a finite positive measure $\nu$ on $K=\overline{\Omega}$, define projection $P_\nu$ (following the line of the proof of \cite[Theorem 4.11]{BGLMPS26b}) as in Lemma~\ref{lem:fibrewise-orthogonal-projections},
just $\pi(x)$, instead being the orthogonal projection, equals $\widetilde{P}(\linspan\{\uu_n(x);n\in\N\})$. These projections may be glued to a norm-one projection $P:M(K,X)\to M_{\mathrm{rep}}(K,X)$ as in \cite[Theorem 4.11]{BGLMPS26b} (with the analogous changes to those indicated in the proof of Proposition~\ref{prop:glue-Pnu}). We conclude by repeating the proof of Proposition~\ref{prop:dual-complementability-transfer}.
    
\end{remark}

\section{Gluing principle for projections}\label{sec:gluing}

In this section we extend the result from the previous section to general, possibly unbounded,
open convex sets in a separable Hilbert space. This will complete the proof of Theorem~\ref{thm:complement}. The main tool is an abstract principle of gluing projections which is valid for general metric spaces. It is contained in the following theorem.

\begin{theorem}\label{thm:fromBallToTheWholeSpace}
Let $(M,d,0)$ be a complete pointed metric space.  For $0<r<s$, let $\iota_{r,s}:\F(B(0,r))\longrightarrow\F(B(0,s))$ be the canonical isometric inclusion. Suppose that, for every $r>0$, there is a bounded linear projection $Q_r:\F(B(0,r))^{**}\longrightarrow\F(B(0,r))$. Assume $C:=\sup_{r>0}\norm{Q_r}<\infty$ and assume the operators $Q_r$ are compatible in the sense that
\begin{equation}\label{eq:metric-ball-coherent-retractions}
 Q_s\iota_{r,s}^{**}=\iota_{r,s}Q_r
 \qquad\mbox{ for }0<r<s.
\end{equation}
Then $\F(M)$ is $C$-complemented in its bidual.
\end{theorem}

\begin{remark}
 (1) The above theorem holds with formally weaker assumptions -- it is enough to assume that there is a sequence $(r_n)$ of radii such that $r_n\nearrow\infty$ which satisfies the respective conditions. The proof is essentially the same, only the sequence $(2^{n+1})$ used in the proof should be replaced by a suitable subsequence of $(r_n)$.

 (2) Note that \eqref{eq:metric-ball-coherent-retractions} easily implies that $\norm{Q_r}\le \norm{Q_s}$ for $0<r<s$, so $C$ may be expressed also as the limit of $\norm{Q_r}$.
\end{remark}

\begin{proof} If $M$ is bounded, the statement is trivial. So, assume that $M$ is unbounded. We will use a technique from \cite{AliagaPernecka} to glue the projections. The proof will be done in several steps.

\smallskip

{\bf Step 1.} A sequence $(W_n)$ of multiplication operators on $\F(M)$.

\smallskip

To shorten the notation set $r_n=2^{n+1}$ for $n\in\N$.  Let $\eta:[0,\infty)\to [0,1]$ be such that $\eta=1$ on $[0,1]$, $\eta=0$ on $[2,\infty)$ and $\eta$ is linear on $[1,2]$. Then $\eta$ is $1$-Lipschitz and hence, for each $n\in\N$, the function 
$\eta_n(t)=\eta(2^{-n}t)$, $t\in [0,\infty)$ is $2^{-n}$-Lipschitz.

Next define $w_n(x)=\eta_n(d(x,0))$ for $x\in M$. Then $\norm{w_n}_\infty=1$, $w_n$ is $2^{-n}$-Lipschitz and $w_n=0$ outside $B(0,r_n)$. Denote by $W_n$ the operator $W_{w_n}$ from Lemma~\ref{lem:free-weighting}. The quoted lemma gives, among others, that $\norm{W_n}\le 3$ and that $W_n^*:\Lip_0(M)\to\Lip_0(M)$ is given by $W_n^*f=w_nf$ for $f\in \Lip_0(M)$.

\smallskip

{\bf Step 2.} A sequence of operators $L_n:\F(M)^{**}
\to\F(M)$.

\smallskip

    Let $J_n:\F(B(0,r_n))\to \F(M)$ be the canonical isometric embedding. Then $J_n^*$ is an isometric quotient of $\F(M)^*=\Lipz(M)$ onto $\F(B(0,r_n))^*=\Lipz(B(0,r_n))$ (cf. Proposition~\ref{prop:subspaces-quotients}). Moreover, it is given by the restriction $J_n^*f=f|_{B(0,r_n)}$ for $f\in \Lipz(M)$. Further, $J_n^{**}$ is again an isometric inclusion (cf. Proposition~\ref{prop:subspaces-quotients}), this time of $\F(B(0,r_n))^{**}$ into $\F(M)^{**}$. Its range is $(\ker J_n^*)^\perp$. Set
    \[
 L_n  :=
 J_n\circ Q_{r_n}\circ (J_n^{**})^{-1}\circ W_n^{**}
 \]
 Then $L_n$ is a well-defined bounded linear operator from $\F(M)^{**}$ to $\F(M)$. Indeed, if $\xi^{**}\in \F(M)^{**}$ and $f\in \ker J_n^*$, then
 $$\pair{W_n^{**}\xi^{**}}{f}=\pair{\xi^{**}}{W_n^*f}=\pair{\xi^{**}}{w_nf}=\pair{\xi^{**}}{0}=0,$$
 where we used that $w_n$ is supported by $B(0,r_n)$ and $f|_{B(0,r_n)}=J_n^*f=0$. Thus the range of $W_n^{**}$ is contained in the range of $J_n^{**}$, hence we may apply the inverse. Clearly, $\norm{L_n}\le \norm{Q_{r_n}}\norm{W_n}\le 3C$.

\smallskip

{\bf Step 3.} The limit operator $Q:\F(M)^{**}\to\F(M)$ is defined by $Q\xi^{**}=\lim_n L_n\xi^{**}$ for $\xi^{**}\in\F(M)^{**}$.

\smallskip

  We note that the function $w_n$ coincides with function $H_n$ from \cite[equation (2.5)]{AliagaPernecka} and hence operator $W_n^*$ coincides with $\mathcal{T}_{H_n}$ from \cite{AliagaPernecka}. Hence, it follows from the beginning of Section 3.1 in \cite{AliagaPernecka} that $(W_n^{**}\xi^{**})$ converges in norm for any $\xi^{**}\in \Lipz(M)^*=\F(M)^{**}$.

We shall use this fact to prove that, given $\xi^{**}\in \F(M)^{**}$, the sequence $(L_n\xi^{**})$ is norm-Cauchy, hence norm convergent. This will show that $Q$ is a well-defined bounded linear operator.

To this end fix $\xi^{**}\in \F(M)^{**}$ and natural numbers $n<m$. Then $J_m\iota_{r_n,r_m}=J_n$, and hence $J_m^{**}\iota_{r_n,r_m}^{**}=J_n^{**}$. Therefore
$$\begin{aligned}
   L_n &= J_n\circ Q_{r_n}\circ (J_n^{**})^{-1}\circ W_n^{**} = J_m\circ \iota_{r_n,r_m} \circ Q_{r_n}\circ (J_n^{**})^{-1}\circ W_n^{**}
   \\& =  J_m\circ Q_{r_m}\circ \iota_{r_n,r_m}^{**}\circ (J_n^{**})^{-1}\circ W_n^{**}
\end{aligned}$$
Hence,
$$\norm{L_m\xi^{**}-L_n\xi^{**}}\le C\norm{(J_m^{**})^{-1} W_m^{**}\xi^{**}- \iota_{r_n,r_m}^{**}(J_n^{**})^{-1} W_n^{**}\xi^{**}}.$$
Taking into account that $J_m^{**}$ is an isometric embedding, we get
$$\begin{aligned}
    \norm{(J_m^{**})^{-1} W_m^{**}\xi^{**}- \iota_{r_n,r_m}^{**}(J_n^{**})^{-1} W_n^{**}\xi^{**}}&=\norm {W_m^{**}\xi^{**}- J_m^{**}\iota_{r_n,r_m}^{**}(J_n^{**})^{-1} W_n^{**}\xi^{**}}\\&=\norm{W_m^{**}\xi^{**}-W_n^{**}\xi^{**}},\end{aligned}$$
so $$\norm{L_m\xi^{**}-L_n\xi^{**}}\le C\norm{W_m^{**}\xi^{**}-W_n^{**}\xi^{**}}.$$
Since $(W_n^{**}\xi^{**})$ is convergent (by the previous paragraph), hence Cauchy, we deduce that $(L_n\xi^{**})$ is also Cauchy, hence convergent. This completes the argument.

\smallskip

{\bf Step 4.} $Q$ is a projection.

\smallskip

 Fix a finitely supported element $\xi\in\F(M)$. Fix $m\in\N$ such that the support of $\xi$ is contained in $B(0,2^m)$.
 For $n\ge m$ we then have
$$\begin{aligned} 
L_n\xi&= J_n\circ Q_{r_n}\circ (J_n^{**})^{-1}\circ W_n^{**}\xi = J_n\circ Q_{r_n}\circ (J_n^{**})^{-1}\circ W_n\xi\\&=J_n\circ Q_{r_n}\circ (J_n^{**})^{-1} \xi =\xi,
  \end{aligned} $$
where $W_n\xi=\xi$ because $w_n=1$ on the support of $\xi$. We deduce $Q\xi=\xi$. Since finitely supported elements are dense in $\F(M)$, we deduce that $Q$ is a projection of $\F(M)^{**}$ onto $\F(M)$.

\smallskip

{\bf Step 5.} $\norm{Q}\le C$

 Note that the estimates from Step 2 give easily $\norm{Q}\le 3C$. To improve it we will use another result of \cite{AliagaPernecka}. As mentioned in Step 3, sequence $(W_n^{**}\xi^{**})$ converges in norm for any $\xi^{**}\in \F(M)^{**}$. So, we may define an operator $T:\F(M)^{**}\to\F(M)^{**}$ by $T\xi^{**
}=\lim_n W_n^{**}\xi^{**}$ for $\xi^{**}\in \F(M)^{**}$. It follows from \cite[Corollary 3.4]{AliagaPernecka} that $\norm{T}\le 1$. Now it may be easily deduced that $\norm{Q}\le C$. This completes the proof.
\end{proof}

We continue by showing that projections constructed in Proposition~\ref{prop:dual-complementability-transfer} are compatible in the sense of Theorem~\ref{thm:fromBallToTheWholeSpace}. This will complete the proof of Theorem~\ref{thm:complement}.

\begin{proposition}\label{prop:compatible}
    Let $X$ be a separable Hilbert space and let $\Gamma$ be a fixed non-degenerate Gaussian measure on $X$. Let $\Omega_1$ and $\Omega_2$ be two bounded open convex sets such that $0\in\Omega_1\subset \Omega_2$. Let $\iota$ denote the canonical isometric embedding of $\F(\Omega_1)$ into $\F(\Omega_2)$. For $j=1,2$ let $Q_j:\F(\Omega_j)^{**}\to\F(\Omega_j)$ be the projection provided by Proposition~\ref{prop:dual-complementability-transfer}. Then $\iota Q_1=Q_2\iota^{**}$.
\end{proposition}

\begin{proof} The proof will be done in several steps.

\smallskip

{\bf Step 1:} The setting.

\smallskip

    For $j=1,2$ let $K_j=\overline{\Omega_j}$, $\gamma_j$, $S_j$, $\Phi_j$, $R_j$, $\theta^j$, $q_j$, $P_j$ be the respective variants of the above-defined objects.
    Further, let $J:L^1(\gamma_1;X)\to L^1(\gamma_2;X)$ be the isometric embedding defined by $$\f\in L^1(\gamma_1;X) \mapsto \frac{\Gamma(\Omega_2)}{\Gamma(\Omega_1)}\f\mbox{ extended by zero on }\Omega_2\setminus\Omega_1.$$ It easily follows from the formula from Theorem~\ref{thm:main} that $S_2J=\iota S_1$.
    We recall that, similarly as in Section~\ref{sec:bounded} we regard $X$ as the dual of $Y=X$.

\smallskip   

    {\bf Step 2.} Let $\bmu\in M(K_2,X)$. Assume $\bmu$ is carried by $K_1$ and representable. Then $q_2(\bmu)\in\iota(\F(K_1))$.

\smallskip

    To locate the representative of $\boldsymbol\mu$, we recall two properties of  supports in Lipschitz-free spaces. For a complete pointed metric space $M$ and $\xi\in\F(M)$, a point $x\in M$ belongs to the support of $\xi$, denoted by
$\supp\xi$, if and only if every neighbourhood $U$ of $x$ supports some $f\in\Lipz(M)$ with $\pair{\xi}{f}\ne0$. Moreover, for every closed $A\subset M$ containing $0$, $\supp\xi\subset A$ if and only if
$\xi\in\iota_{A,M}(\F(A))$, where $\iota_{A,M}$ denotes the canonical inclusion; see \cite[Propositions~2.6 and 2.7]{Aliaga-supports}.

Let $\xi=q_2(\bmu)\in\F(K_2)$. We will show that the support of $\xi$ is contained in $K_1$. To this end fix $x\in K_2\setminus K_1$. Fix $y\in Y$ and $\alpha,\beta>0$ such that
    $$ \pair{x}{y} > \alpha>\beta>\sup_{u\in K_1} \pair{u}{y}.$$
    This is a consequence of the Hahn-Banach separation theorem. Fix a $C^\infty$-function $g:\R\to\R$ such that $g=0$ on $(-\infty,\beta]$ and $g=1$ on $[\alpha,\infty)$. Set $\varphi(u)=g(\pair{u}{y})$ for $u\in K_2$. Then $\varphi\in \Cyl_Y^0(K_2)$. Moreover, $\varphi\bmu=0$ (as $\bmu$ is supported by $K_1$). Hence also $q_2(\varphi\bmu)=0$. We deduce from Proposition~\ref{prop:dual-module} that $W_\varphi(\xi)=0$ as well.
    (Indeed, we may use \eqref{eq:dual-module-representative} -- the left-hand side is $q_2(\varphi\bmu)=0$, the integral on the right-hand side is also $0$ because $\bmu$ is supported by $K_1$ and $\varphi'=0$ on $K_1$, hence the remaining term $W_\varphi(q_2(\bmu))=W_\varphi(\xi)$ should be zero as well.)
    
    Let now $f\in\Lipz(K_2)$ be such that its support is contained in $\{u\in K_2;\pair{u}{y}>\alpha\}$. Then
    $W_\varphi^*f=\varphi f=f$ and hence
    $$\pair{f}{\xi}=\pair{W_\varphi^*f}{\xi}=\pair{f}{W_\varphi \xi}=\pair{f}{0}=0.$$
    This shows that $x\notin \supp \xi$. Thus $\supp \xi\subset K_1$, hence $\xi\in\iota(\F(K_1))$. This completes the argument.

\smallskip

{\bf Step 3.} Fix $\xi^{**}\in \F(\Omega_1)^{**}$ and find $\g^{**}\in (L^1(\gamma_1;X))^{**}$ with $S_1^{**}\g^{**}=\xi^{**}$.  Set $\bmu_1=R_1\g^{**}$ and $\bmu_2=R_2J^{**}\g^{**}$.
Then $\bmu_2$ is carried by $K_1$ and its restriction to $K_1$ coincides with $\bmu_1$.  

\smallskip

    We have
    $$\iota^{**}\xi^{**}=\iota^{**}S_1^{**}\g^{**}=S_2^{**}J^{**}\g^{**},$$
    where the last equality follows from equality $\iota S_1=S_2 J$ mentioned in Step 1.
    For $\f\in C(K_2,Y)$ we have
    $$\pair{\bmu_2}{\f}=\pair{R_2J^{**}\g^{**}}{\f}=\pair{J^{**}\g^{**}}{\Phi_2(\f)}=\pair{\g^{**}}{J^*\Phi_2(\f)}.$$
    Since for $\g\in L^1(\gamma_1;X)$ we have
    $$\begin{aligned}
            \pair{J^*\Phi_2(\f)}{\g}&=\pair{\Phi_2(\f)}{J\g}=
    \frac{\Gamma(\Omega_2)}{\Gamma(\Omega_1)}\int_{\Omega_1}
\pair{\g(x)}{\f(x)}\,d\gamma_2(x)\\&=\int_{\Omega_1}
\pair{\g(x)}{\f(x)}\,d\gamma_1(x)=\pair{\Phi_1(\f|_{K_1})}{\g},\end{aligned}$$ 
we deduce that $J^*\Phi_2(\f)=\Phi_1(\f|_{K_1})$ and hence we may complete the above computation to get

$$\pair{\bmu_2}{\f}=\pair{\g^{**}}{\Phi_1(\f|_{K_1})}=\pair{R_1\g^{**}}{\f|_{K_1}}=\pair{\bmu_1}{\f|_{K_1}},\quad \f\in C(K_2,Y).$$
It follows that $\bmu_2$ is carried by $K_1$ and $\bmu_1$ is the restriction of $\bmu_2$ as required (note that functions from $C(K_1,Y)$ are exactly restrictions of functions from $C(K_2,Y)$).

\smallskip

{\bf Step 4.} Let $\nu_1=\abs{\bmu_1}$ and $\nu_2=\abs{\bmu_2}$. By Step 3 we know that $\nu_2$ is carried by $K_1$ and $\nu_1$ is the restriction of $\nu_2$ to $K_1$. Hence there is also a canonical isometric identification of  $L^1(\nu_1;X)$ and $L^1(\nu_2;X)$. Given $\f\in L^1(\nu_1;X)=L^1(\nu_2;X)$, we have that $\f\nu_1$ is representable if and only if $\f\nu_2$ is representable. Moreover, in such a case, $q_2(\f\nu_2)=\iota(q_1(\f\nu_1))$.

\smallskip

Assume first that $\f\nu_2$ is representable and let $\xi_2=q_2(\f\nu_2)$. By Step 2 we get $\xi_2=\iota \xi_1$ for some $\xi_1\in\F(K_1)$. Then for each  $\varphi\in\Cyl_Y^0(K_2)$ we have
$$ \int_{K_1} \pair{\f(x)}{\varphi'(x)}\, d\nu_1(x)=\int_{K_2} \pair{\f(x)}{\varphi'(x)}\, d\nu_2(x)=\pair{\varphi}{\xi_2}=\pair{\varphi}{\iota \xi_1}=\pair{\varphi|_{K_1}}{\xi_1}.$$
Since functions from $\Cyl_Y^0(K_1)$ are exactly restrictions of functions from $\Cyl_Y^0(K_2)$, we get that $\f\nu_1$ is representable and $q_1(\f\nu_1)=\xi_1$.

Conversely, if $\f\nu_1$ is representable and $\xi_1=q_1(\f\nu_1)$, then it is easy to check (using similar computation) that $\f\nu_2$ is representable and $q_2(\f\nu_2)=\iota(\xi_1)$.

\smallskip

{\bf Step 5.} $q_2P_2\bmu_2=\iota q_1P_1\bmu_1$.

\smallskip

Assume that $\bmu_2=\uu\nu_2$, for some $\uu\in L^1(\nu_2;X)$. It follows from Step 3 that
$\bmu_1=\uu\nu_1$, using the identification from Step 4. It follows from Step 4 that $M_{\nu_1}=M_{\nu_2}$, hence their associated closed subspaces of $X$ coincide almost everywhere, by Proposition~\ref{prop:tangent-fields}, and so do the orthogonal projections onto them. Therefore $P_{\nu_1}(\uu)=P_{\nu_2}(\uu)$ (where $P_{\nu_1}$ and $P_{\nu_2}$ are the projections from Lemma~\ref{lem:fibrewise-orthogonal-projections}). Hence $P_2(\bmu_2)$ is carried by $K_1$ and its restriction to $K_1$ is $P_1(\bmu_1)$. Thus the equality follows from the `moreover' part of Step 4.

\smallskip

{\bf Step 6.} Conclusion.

\smallskip

By Steps 3--5 we get
$$q_2P_2R_2J^{**}=\iota q_1P_1R_1.$$ 
 The defining identities $Q_jS_j^{**}=q_jP_jR_j$ from Proposition~\ref{prop:dual-complementability-transfer} together with equality $\iota S_1=S_2J$ from Step 1 then yield
 \[
 \iota Q_1S_1^{**}=Q_2S_2^{**}J^{**}=
 Q_2\iota^{**}S_1^{**}
\]
Since $S_1^{**}$ is surjective, we conclude that $Q_2\iota^{**}=\iota Q_1$.
\end{proof}

\begin{remark}\label{rem:twodim-compatible}
    An analogue of the previous proposition holds if $X$ is a general Banach space of dimension two. In this case projections $Q_1$ and $Q_2$ are provided by Remark~\ref{rem:dim-bounded} instead of Proposition~\ref{prop:dual-complementability-transfer}. The proof is the same, the only difference is in Step 5. To prove that $P_{\nu_1}(\uu)=P_{\nu_2}(\uu)$ we use that the assignment $\widetilde{P}$ from \cite[Proposition 4.4]{BGLMPS26b} recalled in Remark~\ref{rem:dim-bounded} depends only on $X$, not on $\Omega$.
\end{remark}


\section{Proof of Theorem~\ref{thm:bidual-complemented-implies-Hilbert}}
\label{sec:Mrep-complemented-implies-Hilbert}

The aim of this section is to prove Theorem~\ref{thm:bidual-complemented-implies-Hilbert} which will complete the proof of Theorem~\ref{thm:Intro1}.
It will be done in several steps. The starting point is a bounded linear projection $Q:\F(M)^{**}\to \F(M)$. Given a finite-dimensional subspace $V\subset X$, we aim to find a projection $P:X\to V$ with $\norm{P}\le \norm{Q}$. The first step is a finite-dimensional reduction provided by the general principle of local complementation contained in the lemma below. To formulate it in an efficient way we introduce the following notation.

Given a finite-dimensional subspace $V$ of a Banach space $X$, we denote its \emph{relative projection constant} by
\begin{equation}\label{eq:relative-projection-constant}
 \lambda(V,X)=\inf\{\norm P:P:X\to V\text{ is a bounded  linear projection}\}.
\end{equation}
This number is finite by a consequence of the Hahn--Banach theorem. 

\begin{lemma}\label{lem:projection-constant-local}
Let $X$ be a Banach space and let $V\subset X$ be a nonzero
finite-dimensional subspace. Then there is a bounded linear projection $P:X\to V$ such that
\[
 \norm P = \lambda(V,X)
 =\sup\{\lambda(V,E):V\subset E\subset X,\ \dim E<\infty\}.
\]
\end{lemma}

\begin{proof}
Denote the supremum by $C$. Restricting projections gives
$C\le\lambda(V,X)$.

Given a finite-dimensional subspace $E\subset X$ containing $V$, let $(P_n)$ be a sequence of projections of $E$ onto $V$ with $\norm{P_n}\to\lambda(V,E)$. Since the space of operators
$L(E,V)$ has finite dimension, the sequence $(P_n)$  has a (norm-)convergent subsequence.
Denote its limit by $P_E$. It is a projection of norm $\lambda(V,E)$, in particular $\norm{P_E}\le C$. 

Let $\widetilde{P_E}:X\to V$ be a mapping which equals $P_E$ on $E$ and $0$ on $X\setminus E$.
Consider the net $(\widetilde{P_E}|_{B_X})$, indexed by the directed set of all finite-dimensional subspaces of $X$ containing $V$, ordered by inclusion. These restrictions may be regarded as elements of the product space $(CB_V)^{B_X}$, which is compact by the Tychonoff theorem. Hence, we may find a convergent subnet and let $\widetilde{P}$ be its limit. It is easy to check that $\widetilde{P}$ may be (uniquely) extended to a linear projection $P:X\to V$ with $\norm{P}\le C$. This completes the proof.
\end{proof}

The above lemma enables us to reduce the proof of Theorem~\ref{thm:bidual-complemented-implies-Hilbert} to a construction of a projection between finite-dimensional spaces. The rest of this section is devoted to this proof. It is divided into several steps, each step in one of the following subsections.

\subsection{The setting and the basic structure of the proof}

We assume that $X$ is a nontrivial Banach space and that $M\subset X$ is a subset with nonempty interior. We further have a bounded linear projection $Q:\F(M)^{**}\to \F(M)$. Set $C=\norm{Q}$. We may assume, without loss of generality, that $0\in \Int M$ (as any translation is an isometry). We may further assume that $U_X\subset M$. Indeed, if $\alpha>0$, then $\F(\alpha M)$ is isometric to $\F(M)$.

We fix nonzero finite-dimensional subspaces $V\subset E\subset X$. We are looking for a projection from $E$ onto $V$ with norm bounded by $C$. Let $\Gamma$ be a non-degenerate Gaussian probability on $V$ and let $\gamma$ denote its normalized restriction to $U_V$. Let $S:L^1(\gamma;V)\to\F(U_V)$ be the operator provided by Theorem~\ref{thm:main}.

The construction of the sought projection is illustrated by the following diagram.

\begin{equation}\label{eq:diagramvelky}
\xymatrix{E \ar[r] & L^1(\gamma;E) \ar[r]^{J}& \F(U_E)^{**} \ar[r] & \F(M)^{**} \ar[d]_{Q}  \\
V \ar[u] \ar[r] & L^1(\gamma;V) \ar[u] \ar[r]^{S} & \F(U_V) \ar[u] \ar[r] & \F(M) \ar[d] \\
&& X \ar[r]^{\delta_X} \ar[rd]^{\varphi} & \F(X) \ar[d]^{\linearization{\varphi}} \\
&&& V
}\end{equation}

The upward arrows are canonical (isometric) embeddings. The first two are canonical inclusions of $V$ into $E$ and of $L^1(\gamma;V)$ into $L^1(\gamma;E)$. The third one is the composition of the canonical embedding of $\F(U_V)$ into $\F(U_E)$ with the embedding of $\F(U_E)$ into its bidual.
The arrow $E\to L^1(\gamma;E)$ is a canonical isometric embedding defined by $x\mapsto w\cdot x$ for a suitable weight function $w\in L^1(\gamma)$ (see Proposition~\ref{prop:projekce-konstrukce} below). The first arrow in the second line is the restriction of this embedding to $V$.

$S$ is, of course, the above-mentioned operator from Theorem~\ref{thm:main} and $J$ is its norm-one `extension' constructed in Proposition~\ref{prop:local-bidual-extension} below. The third arrow in the first line is the second adjoint of the canonical embedding $\F(U_E)\to \F(M)$ induced by the inclusion $U_E\subset M$. The third arrow in the second line is the isometric embedding induced by the inclusion $U_V\subset M$.
$Q$ is the assumed projection. We note that the upper part of the diagram commutes (because $Q$ is a projection).

The embedding $\F(M)\to\F(X)$ is induced by the inclusion $M\subset X$.
Finally, $\varphi$ is a suitably chosen Lipschitz mapping and $\linearization{\varphi}$ is its linearization. 

The sought projection $E\to V$ will be the composition of the horizontal arrows in the first line and the vertical arrows in the last column. In order to obtain a projection with the required norm,
we need to describe $J$ and to find a suitable Lipschitz mapping $\varphi:X\to V$ and a suitable weight $w$. It will be done gradually in the following subsections.

\subsection{Construction of the operator $J$} Our next aim is to construct a suitable operator $J:L^1(\gamma;E)\to \F(U_E)^{**}$. It is the content of Proposition~\ref{prop:local-bidual-extension} below. But first we give the following lemma.

\begin{lemma}\label{lem:gauss-ac}
    The measure $\gamma$ is absolutely continuous with respect to the Lebesgue measure on $V$.
\end{lemma}

\begin{proof}
    Since $V$ has finite dimension, we may assume that $V=\R^d$ for some $d\ge 1$. If $\gamma$ is the normalized restriction of a non-degenerate Gaussian measure, it is absolutely continuous by the very definition. But we will show that the absolute continuity follows only from property~\ref{NT}. Indeed, fix a Lebesgue-null Borel set $A\subset V$.
    Fubini's theorem yields
$$\begin{aligned}
    \int_V\Gamma(A-h)\,dh &=\int_V\int_V\mathbf1_A(x+h)\,d\Gamma(x)\,dh=\int_V\int_V\mathbf1_A(x+h)\,dh\,d\Gamma(x)\\&=\int_V\mathcal{L}^d(A-x)\,d\Gamma(x)=0.\end{aligned}$$ Thus some translate of $A$ is $\Gamma$-null, and property~\ref{NT} yields $\Gamma(A)=0$ (having in mind that $H_\Gamma=V$ as $V$ has finite dimension).
\end{proof}

\begin{proposition}\label{prop:local-bidual-extension}
There is a linear contraction $J:L^1(\gamma;E)\to\F(U_E)^{**}$ such that $J|_{L^1(\gamma;V)} = S$.
\end{proposition}

\begin{proof}
    Since $E$ has finite dimension and $V$ is a subspace of $E$, we may assume that $E=(\R^{d_1+d_2},\norm{\cdot})$ and $V=\{(x,0)\; x\in \R^{d_1}\}\subset E$, where $d_1,d_2\ge1$.

Let $(u_n)$ be an approximate unit in $\R^{d_1+d_2}=E$ as in Section~\ref{subsec:convolution}, so that $\supp u_n\subset n^{-1}U_E$. For each $n\in\N$ we define an operator $J_n:L^1(\gamma;E)\to\F(E)^{**}$ by
$$\pair{J_n\g}{f} = \int_{U_V} \pair{\g(t)}{(u_n*f)'(t)} \, d\gamma(t)= \int_{U_V} \pair{\g(t)}{(u_n*f')(t)} \, d\gamma(t)$$
for $\g\in L^1(\gamma;E)$ and $f\in\Lipz(E)=\F(E)^*$. By Fact~\ref{fact:convolution-approximation}, $J_n$ is a well-defined linear contraction.

Let us prove two properties of the sequence $(J_n)$:

\smallskip

(1) If $f\in \Lipz(E)$ is such that $f|_{U_E}=0$, then $\pair{J_n\g}{f}\to0$ for each $\g\in L^1(\gamma;E)$.

\smallskip

Indeed, $f'=0$ on $U_E$. For $n\ge2$,
$x\in(1-1/n)U_E$ and $z\in\supp u_n$ imply $x-z\in U_E$, so $u_n*f'=0$ on $(1-1/n)U_E$. Hence
$$\begin{aligned}
\abs{\pair{J_n\g}{f}}&=  \abs{\int_{U_V} \pair{\g(t)}{(u_n*f')(t)} \, d\gamma(t)}
=\abs{ \int_{U_V\setminus(1-\frac1n)U_V} \pair{\g(t)}{(u_n*f')(t)} \, d\gamma(t)}
\\&\le  \int_{U_V\setminus(1-\frac1n)U_V} \norm{\g(t)}\norm{(u_n*f')(t)} \, d\gamma(t)\\&\le \norm{f}_{\mathrm{Lip}} \int_{U_V\setminus(1-\frac1n)U_V} \norm{\g(t)} \, d\gamma(t) \to0
\end{aligned}$$
since $\g\in L^1(\gamma;E)$.

\smallskip

(2) If $\g\in L^1(\gamma;V)$, then $J_n\g\to S\g$ in the weak$^*$ topology of $\F(E)^{**}$.

\smallskip

Indeed, pick $\g\in L^1(\gamma;V)$ and fix $f\in \Lipz(E)$. Then
$$\begin{aligned}
    \pair{J_n\g}{f}&=\int_{U_V} \pair{\g(t)}{(u_n*f)'(t)} \, d\gamma(t)
=\int_{U_V} \pair{\g(t)}{(u_n*f)'(t)|_V} \, d\gamma(t)\\& =\int_{U_V} \pair{\g(t)}{((u_n*f)|_V)'(t)} \, d\gamma(t).\end{aligned}$$

 We claim that $((u_n*f)|_V)'\to (f|_V)'$ in the weak$^*$ topology of $L^\infty(V,V^*)$. By Fact~\ref{fact:convolution-approximation}, $u_n*f\to f$ uniformly on $E$ and
$\norm{((u_n*f)|_V)'}_\infty\le\norm f_{\Lip}$. For $\varphi=(\varphi_1,\ldots,\varphi_{d_1})
\in C_c^\infty(V;V)$, integration by parts gives
\[\begin{gathered}
 \int_V\pair{\varphi(s)}{((u_n*f)|_V)'(s)}\,ds
 =-\int_V(u_n*f)(s)
   \sum_{j=1}^{d_1}\frac{\partial\varphi_j}{\partial s_j}(s)\,ds\\
 \longrightarrow
 -\int_V f(s)
   \sum_{j=1}^{d_1}\frac{\partial\varphi_j}{\partial s_j}(s)\,ds =\int_V\pair{\varphi(s)}{(f|_V)'(s)}\,ds.
 \end{gathered}
\]
The uniform bound on the derivatives and density of
$C_c^\infty(V;V)$ in $L^1(V;V)$ therefore imply that
$((u_n*f)|_V)'\to(f|_V)'$ in the weak$^*$ topology of
$L^\infty(V;V^*)$.

Since $\gamma$ is absolutely continuous with respect to the Lebesgue measure (by Lemma~\ref{lem:gauss-ac}), the function
$\g\cdot u$, where $u$ is the density of $\gamma$ with respect to the Lebesgue measure, extended by zero outside $U_V$, belongs to $L^1(V;V)$. Testing the weak$^*$ convergence against this function yields
$$\pair{J_n\g}{f}\to \int_{U_V} \pair{\g(t)}{(f|_V)'(t)}\, d\gamma(t)=\pair{S\g}{f|_V},$$
which proves (2).

\smallskip

Let us now pass to finding $J$. The operator $J$ will be a cluster point of the sequence $(J_n)$ in the weak$^*$ operator topology. A possible interpretation is the following: fix $\mathcal{U}$ a free ultrafilter over $\N$ and define
$$\pair{J\g}{f}=\lim_{\mathcal{U}} \pair{J_n\g}{f},\quad \g\in L^1(\gamma;E), f\in \Lipz(E).$$
Then $J$ is clearly a contractive linear operator from $L^1(\gamma;E)$ into $\F(E)^{**}$.

Property (1) implies that $\pair{J\g}{f}=0$ whenever $\g\in L^1(\gamma;E)$ and $f\in\Lipz(E)$ satisfies
$f|_{U_E}=0$. Thus the range of $J$ is contained in $\F(U_E)^{\perp\perp}$, canonically identified with
$\F(U_E)^{**}$. Finally, property (2) implies that for any $\g\in L^1(\gamma;V)$ we have
$$\pair{J\g}{f}=\pair{S\g}{f|_V},\quad f\in\Lipz(E),$$
which means that $J$ is an extension of $S$ as required.
\end{proof}

\begin{corollary}
    There is a linear operator $H:L^1(\gamma;E)\to \F(X)$ such that $\norm{H}\le C$ and $H|_{L^1(\gamma;V)}=S$.
\end{corollary}

\begin{proof}
    It is enough to compose $J$ with $Q$ with the use of the isometric inclusions from diagram \eqref{eq:diagramvelky}.
\end{proof}

\subsection{Construction of a projection starting from a Lipschitz mapping}

Having constructed the operator $J$ and subsequently $H$, we may now simplify diagram \eqref{eq:diagramvelky} to the following one:

\begin{equation}\label{eq:diagrammaly}
\xymatrix{&& X \ar[d]^{\delta_X} \ar[rd]^{\varphi} & \\
E \ar[r] & L^1(\gamma;E) \ar[r]^{H}& \F(X) \ar[r]_{\linearization{\varphi}} & V  \\
V \ar[u] \ar[r] & L^1(\gamma;V) \ar[u] \ar[r]^{S} & \F(U_V) \ar[u]  & 
}\end{equation}

Here $H$ is the composed operator from the previous subsection. Recall that $\norm{H}\le C$ and $H$ extends $S$. Again, the upwards arrows are canonical inclusions. To define the arrow $E\to L^1(\gamma;E)$ fix a weight function $w\in L^1(\gamma)$ with $w\ge 0$ and $\norm{w}_1=1$. The arrow is then defined by $x\mapsto w\cdot x$. It is clearly an isometric embedding of $E$ into $L^1(\gamma;E)$. The arrow $V\to L^1(\gamma;V)$ is the respective restriction, hence it is given by the same formula.

Finally, $\varphi$ is a suitable pointed Lipschitz mapping and $\linearization{\varphi}$ is its linearization. Next we observe that, under some conditions on $\varphi$ and $w$ we obtain a projection $P:E\to V$.

\begin{proposition}\label{prop:projekce-konstrukce}
    Let $w\in L^1(\gamma)$ be a non-negative function of norm one. Let $\varphi:X\to V$ be a pointed Lipschitz mapping and let $\linearization{\varphi}:\F(X)\to V$ be its linearization.
    Define
    $$Px=\linearization{\varphi}(H(w\cdot x)),\quad x\in E.$$
    Then $P$ is a bounded linear operator from $E$ to $V$. Moreover, if
    $$(\varphi|_{U_V})'(t)=\operatorname{Id}_V\mbox{ for $\gamma$-almost all $t\in U_V$ satisfying }w(t)>0,$$
    then $P$ is a projection of $E$ onto $V$.
\end{proposition}

\begin{proof}
    The first statement is obvious -- $P$ is a bounded linear operator, being a composition of three bounded linear operators. We also get a trivial estimate $\norm{P}\le \norm{H}\norm{\linearization{\varphi}}\le C\norm{\varphi}_{\mathrm{Lip}}$.

    To prove the second statement, let us compute $Px$ for $x\in V$. So, fix $x\in V$ and $x^*\in V^*$ and compute
    $$\begin{aligned}
      x^*(Px)&=x^*(\linearization{\varphi}(H(w\cdot x)))=(x^*\circ\linearization{\varphi})(H(w\cdot x)) =\pair{ H(w\cdot x)}{x^*\circ\varphi} \\& =\pair{ S(w\cdot x)}{x^*\circ\varphi} =\int_{U_V} \pair{w(t)x}{(x^*\circ\varphi)'(t)}\,d\gamma(t)
      \\ & =\int_{U_V} w(t)\pair{x}{(x^*\circ\varphi)'(t)}\,d\gamma(t)
      =\int_{U_V} w(t)\cdot\partial_x(x^*\circ\varphi)(t)\,d\gamma(t)
       \\& =\int_{U_V} w(t)\cdot x^*\circ\partial_x\varphi(t)\,d\gamma(t)
       =  x^*\Big(\int_{U_V} w(t)\cdot\partial_x\varphi(t)\,d\gamma(t)\Big).
    \end{aligned}$$

The first two equalities follow from definitions. Let us explain the third equality. Recall that $H(w\cdot x)\in \F(X)$. Further, $x^*\circ\linearization{\varphi}\in\F(X)^*=\Lipz(X)$. The Lipschitz function representing the functional is given by composition with $\delta_X$; this composition is exactly $x^*\circ\varphi$. The fourth equality follows from the fact that $H$
extends $S$ (using that $x\in V$). The next equality follows from the formula for $S$ from Theorem~\ref{thm:main}. The next three equalities follow by elementary computation. The last one follows from the properties of the Bochner integral.

So, having in mind that $x^*\in V^*$ was arbitrary, we get
$$Px=\int_{U_V} w(t)\cdot\partial_x\varphi(t)\,d\gamma(t),\quad x\in V.$$
Now observe that the condition on $\varphi$ from the statement may be stated as
$$\forall x\in V\colon w(t)\cdot\partial_x\varphi(t)=w(t)\cdot x \quad \gamma\mbox{-almost everywhere on }U_V.$$
This clearly implies that $Px=x$ for $x\in V$, i.e. $P$ is a projection onto $V$.
\end{proof}

We note that the condition from the above proposition is satisfied, for example, if $\varphi$ is a linear projection of $E$ onto $V$. But this is not helpful, because in this way we get only the trivial estimate $\norm{P}\le C\norm{\varphi}$. We point out that we search for a projection $P$ with norm bounded by $C$ (or at least by $C+\varepsilon$ for a given $\varepsilon>0$). If we wanted just some projection (without norm estimate), there would be nothing to prove. So, in the rest of the proof we will show how to reach a small norm of $P$ by a suitable choice of $\varphi$ and $w$.

\subsection{Estimates of the norm of $P$}

In this subsection we are going to provide finer estimates of the norm of $P$ which will be later used to construct a suitable $\varphi$. The first tool is the following lemma.

\begin{lemma}\label{lem:vodhad}
    Let $w$, $\varphi$ and $P$ be as in Proposition~\ref{prop:projekce-konstrukce}. Assume that $A\subset X$ and $\varepsilon>0$ are given such that $H(wB_E)\subset \F(A)+\varepsilon B_{\F(X)}$. Then 
    $$\norm{P}\le (C+\varepsilon)\norm{\varphi|_A}_{\mathrm{Lip}}+\varepsilon\norm{\varphi}_{\mathrm{Lip}}.$$
\end{lemma}

\begin{proof}
Fix $x\in B_E$. Then $H(w\cdot x)=\xi+\varepsilon\zeta$ for some $\xi\in \F(A)$ and $\zeta\in B_{\F(X)}$. Then
$$\norm{Px}=\norm{\linearization{\varphi}(\xi+\varepsilon\zeta)}\le \norm{\linearization{\varphi}|_{\F(A)}}\norm{\xi}+\varepsilon\norm{\linearization{\varphi}}
    =\norm{\varphi|_A}_{\mathrm{Lip}}\norm{\xi}+\varepsilon\norm{\varphi}_{\mathrm{Lip}}.$$
 Moreover,
 $$\norm{\xi}\le \norm{H(w\cdot x)}+\varepsilon\le C+\varepsilon.$$
 By combining the two inequalities we obtain the required estimate.
\end{proof}

The previous lemma will be applied together with the following one: 

\begin{lemma}[{\cite[Theorem A]{APPVUniformRegularity}}]\label{lem:uniform-regularity}
Let $M$ be a complete pointed metric space and let
$W\subset\F(M)$ be relatively weakly compact.
For every closed set $S\subset M$ containing $0$ and every
$\varepsilon>0$, there is a compact set $K\subset M\setminus S$ such that
\[
 W\subset\F(S\cup K)+\varepsilon B_{\F(M)}.
\]
\end{lemma}

Note that the space $L^2(\gamma;E)$ is reflexive (as $E$ is finite-dimensional) and so its unit ball is weakly compact. Therefore the following lemma is an immediate corollary to the previous one.

\begin{lemma}\label{lem:ball of L2}
For every $\varepsilon>0$ there is a compact subset $K_\varepsilon\subset X\setminus V$ such that
$$H(B_{L^2(\gamma;E)})\subset \F(V\cup K_\varepsilon)+\varepsilon B_{\F(X)}.$$
\end{lemma}

We continue by a technical lemma on a precise estimate of the norm of $P$.

\begin{lemma}\label{lem:estimate}
    Fix $\pi:X\to V$ a bounded linear projection. Let $w\in L^1(\gamma)$ be a non-negative function of norm one. Assume, moreover, that $w\in L^2(\gamma)$. Let $\psi:V\to V$ be a pointed bounded Lipschitz map. Then for each $\varepsilon>0$ there is a Lipschitz extension $\varphi:X\to V$ of $\psi$ such that the operator $P$ provided by Proposition~\ref{prop:projekce-konstrukce} satisfies
    $$\norm{P}\le(C+\varepsilon\norm{w}_2)\max\{\norm{\psi}_{\mathrm{Lip}},\tfrac1r\norm{\psi}_\infty\}+\varepsilon\norm{w}_2(\norm{\pi}\cdot\norm{\psi}_{\mathrm{Lip}}+\tfrac1r\norm{\psi}_\infty),$$
    where $r=\dist(V,K_\varepsilon)$.
 \end{lemma}

\begin{proof}
    Fix $\varepsilon>0$ and let $K_\varepsilon$ be the compact set provided by Lemma~\ref{lem:ball of L2}. Then
    $$H(w B_E)\subset H(\norm{w}_2 B_{L^2(\gamma;E)})\subset\norm{w}_2(\F(V\cup K_\varepsilon)+\varepsilon B_{\F(X)})=\F(V\cup K_\varepsilon)+\varepsilon \norm{w}_2 B_{\F(X)}.$$
    
    Let $r=\dist(K_\varepsilon,V)$. Then $r>0$, as $V$ is closed, $K_\varepsilon$ compact and $V\cap K_\varepsilon=\emptyset$.     Let $v:X\to\R$ be defined by
    $$v(x)=\max\{1-\tfrac{1}{r}\dist(x,V),0\},\quad x\in X.$$
    Then $v$ is clearly $\frac1r$-Lipschitz, $v=1$ on $V$, $v=0$ on $K_\varepsilon$ and $v$
    attains values in $[0,1]$. Set
    $$\varphi= v\cdot (\psi\circ\pi).$$
    Then $\varphi$ is a mapping $X\to V$, $\varphi$ extends $\psi$ and $\varphi=0$ on $K_\varepsilon$. It easily follows that
    $$\norm{\varphi|_{V\cup K_\varepsilon}}_{\mathrm{Lip}}\le \max\{\norm{\psi}_{\mathrm{Lip}},\frac1r\norm{\psi}_\infty\}.$$
    Moreover, since $\psi\circ\pi$ is $\norm{\pi}\cdot\norm{\psi}_{\mathrm{Lip}}$-Lipschitz and
    bounded by $\norm{\psi}_\infty$, and $v$ is $\frac1r$-Lipschitz and bounded by $1$, we get that
    $$\norm{\varphi}_{\mathrm{Lip}}\le \norm{\pi}\cdot\norm{\psi}_{\mathrm{Lip}}+\tfrac1r\norm{\psi}_\infty.$$
    Therefore, Lemma~\ref{lem:vodhad} yields the required estimate.
   \end{proof}

\subsection{Construction of a Lipschitz mapping and completion of the proof}

We are now going to construct a Lipschitz mapping $\psi:V\to V$ which will enable us to use Lemma~\ref{lem:estimate} and Proposition~\ref{prop:projekce-konstrukce} to finish the proof.

\begin{lemma}\label{lem:periodic-densities}
For every $\varepsilon>0$ there is $\alpha>0$ such that, for every $r>0$, there are an open set $A\subset V$ and a pointed Lipschitz map $\psi:V\to V$ satisfying
\[
 \gamma(A)\ge\alpha,\qquad
 \norm\psi_{\Lip}\le1+\varepsilon,\qquad
 \norm{\psi}_\infty\le r,\qquad
 \psi'(s)=\operatorname{Id}_V \text{ for every }s\in A.
\]
\end{lemma}

\begin{proof}[Proof of Lemma~\ref{lem:periodic-densities}]
Fix $\varepsilon>0$ and put $R=1+1/\varepsilon$. We proceed in several steps.

\smallskip

\noindent\textbf{Step 1.}
We construct a map $F:V\to V$ which equals the identity on
$B_V$, vanishes outside $RU_V$, and satisfies
$\norm{F}_{\Lip}\le1+\varepsilon$ and 
$\norm{F}_\infty\le R$.

\smallskip

Let $\chi:[0,\infty)\to[0,1]$ be constant $1$ on $[0,1]$, constant zero on $[R,\infty)$ and linear on $[1,R]$. Then $\chi$ is $\frac{1}{R-1}=\varepsilon$-Lipschitz. Further, define
\[ 
 F(v)=\chi(\norm v)\cdot v, \quad v\in V.
\]
Then $F$ clearly equals the identity on $B_V$ and vanishes outside $RB_V$. Obviously, $\norm{F}_\infty\le R$ (in fact, $\norm{F}_\infty=\frac{R^2}{4(R-1)}$ if $\varepsilon<1$).
Next we show that $F$ is $1+\varepsilon$-Lipschitz. To this end we use the explicit formula
$$\chi(t)=\frac{R-t}{R-1},\quad t\in [1,R].$$
For $x,y\in RB_V\setminus U_V$ with $\norm x\ge\norm y$, we have
\[
 \begin{aligned}
 \norm{F(x)-F(y)}
 &\le\chi(\norm x)\norm{x-y}
       +\norm y\,|\chi(\norm x)-\chi(\norm y)|\\
  &=  \frac{R-\norm{x}}{R-1}\norm{x-y} + \norm{y} \cdot\abs{\frac{\norm{x}-\norm{y}}{R-1}}
 \\ &\le\frac{R-\norm x+\norm y}{R-1}\norm{x-y} \le\frac{R}{R-1}\norm{x-y}
 =(1+\varepsilon)\norm{x-y}.
 \end{aligned}
\]
Since $F$ is continuous, equals the identity on $B_V$, and
vanishes outside $RU_V$, a line-segment argument gives
$\norm{F}_{\Lip}\le1+\varepsilon$.

\smallskip
\noindent\textbf{Step 2.}
Choose a maximal $3R$-separated set $\Lambda\subset V$
containing $0$. On each ball $B_V(z,R)$,
$z\in\Lambda$, put $\psi_0(s)=F(s-z)$, and put
$\psi_0(s)=0$ elsewhere. These balls are disjoint,
$\psi_0(0)=0$, and $\norm{\psi_0}_\infty\le R$.
Since each translated copy of $F$ vanishes on the boundary
of its ball, a line-segment argument gives
$\norm{\psi_0}_{\Lip}\le1+\varepsilon$.
Put $G=\bigcup_{z\in\Lambda}U_V(z,1)$. Then $G$ is an open set
and $\psi_0'=\operatorname{Id}_V$ on $G$.

\smallskip
\noindent\textbf{Step 3.} There is a finite set $F\subset V$ such that $F+G=V$.

Indeed, maximality of $\Lambda$ gives $V=\Lambda+3RU_V$.
By compactness of the set $3RB_V$ we may choose $a_1,\ldots,a_N\in V$ such that
$3RB_V\subset\bigcup_{j=1}^N U_V(a_j,1)$. 
Thus it is enough to set $F=\{a_1,\dots,a_N\}$. Then
$$F+G=\bigcup_{z\in\Lambda} \bigcup_{j=1}^N U(z+a_j,1)\supset\bigcup_{z\in\Lambda}  B(z,3R)=V.$$

\smallskip
\noindent\textbf{Step 4.} Given $t>0$ and $a\in F$ define
\[
 \psi_{a,t}(s)=
 \frac{\psi_0(ts-a)-\psi_0(-a)}{t},
 \qquad s\in V.
\]
Then $\psi_{a,t}(0)=0$, it is $1+\varepsilon$-Lipschitz and $\norm{\psi_{a,t}}_\infty\le \frac{2R}{t}$. Moreover, $\psi_{a,t}'=\operatorname{Id}_V$ on the open set 
$$G_{a,t}=\frac1t(a+G).$$

\smallskip
\noindent\textbf{Step 5.} Conclusion.

Set $\alpha=1/N$ (recall that $N$ is the cardinality of $F$). Fix $r>0$. Put $t=2R/r$. By Steps 3 and 4 we have $V=\bigcup_{a\in F}G_{a,t}$. So, there is some $a\in F$ such that $\gamma(G_{a,t})\ge\frac1N=\alpha$. We conclude by setting $\psi=\psi_{a,t}$ and $A=G_{a,t}$.
\end{proof}

\begin{proof}[Synthesis of the proof of Theorem~\ref{thm:bidual-complemented-implies-Hilbert}]
 By Proposition~\ref{prop:projekce-konstrukce} it is enough to show that
 given $\varepsilon>0$ there is a pointed Lipschitz map $\varphi: X\to V$ and a non-negative norm-one weight $w\in L^1(\gamma)$ such that $(\varphi|_{U_V})'(t)=\operatorname{Id}_V$ for $\gamma$-almost all $t\in U_V$ satisfying $w(t)>0$ and, moreover, the operator $P$ provided by Proposition~\ref{prop:projekce-konstrukce} has norm at most $C+\varepsilon$.

So, fix $\varepsilon>0$. Further, fix $\pi:X\to V$ a bounded linear projection.
Let $\alpha>0$ be the number provided by Lemma~\ref{lem:periodic-densities} (corresponding to the chosen $\varepsilon$). Let $\varepsilon'>0$ be such that
$$\varepsilon'<\frac{\varepsilon \alpha^{1/2}}{\norm{\pi}}.$$ 
Let $K_{\varepsilon'}$ be the compact set provided by Lemma~\ref{lem:ball of L2} and set $r=\dist(K_{\varepsilon'},V)$. Let $\psi$ and $A$ be the objects provided by Lemma~\ref{lem:periodic-densities}. Set
$$w=\frac1{\gamma(A)}\chi_A.$$
Then $w$ is a positive weight, $\norm{w}_1=1$ and $\norm{w}_2=\frac{1}{\gamma(A)^{1/2}}\le \frac{1}{\alpha^{1/2}}$.

Let $\varphi:X\to V$ be an extension of $\psi$ provided by Lemma~\ref{lem:estimate} for the parameters $w$ and $\varepsilon'$. Let $P$ be the operator provided by Proposition~\ref{prop:projekce-konstrukce}. Since $(\varphi|_V)'=\psi'$, Proposition~\ref{prop:projekce-konstrukce} yields that $P$ is a projection. Furthermore by the estimate in Lemma~\ref{lem:estimate} we get
$$\begin{aligned}
    \norm{P}&\le(C+\varepsilon'\norm{w}_2)\max\{\norm{\psi}_{\mathrm{Lip}},\tfrac1r\norm{\psi}_\infty\}+\varepsilon'\norm{w}_2(\norm{\pi}\cdot\norm{\psi}_{\mathrm{Lip}}+\tfrac1r\norm{\psi}_\infty)\\
    &\le (C+\varepsilon)\max\{1+\varepsilon,1\}+ \varepsilon(1+\varepsilon)+\varepsilon.
\end{aligned}$$
It is clear that this estimate may be arbitrarily close to $C$. This completes the proof.
\end{proof}

\section{A byproduct of the proof}

In this section we formulate a proposition which was essentially proved within the previous section. We point it out because we are convinced that it may be useful also in further situations.

\begin{proposition}\label{app:lem:projection-extraction}
Let $X$ be a Banach space and let $V$ and $E$ be finite-dimensional linear subspaces with $\{0\}\ne V\subset E\subset X$. Let $\Gamma$ be a non-degenerate Gaussian probability measure on $V$ and let $\gamma$ denote its normalized restriction to $U_V$. Then
\[\begin{aligned}
 \lambda(V,E)=\min\Big\{\norm H:\quad& H:L^1(\gamma;E)\to\F(X) \mbox{ bounded linear},\\ & H|_{L^1(\gamma;V)}=S\Big\},\end{aligned}
\]
where $S$ is the quotient operator from Theorem~\ref{thm:main} and $\F(U_V)$ is canonically identified with a subspace of $\F(X)$.
\end{proposition}

\begin{proof}
    Inequality `$\le$' with infimum in place of minimum follows from the above proof. More precisely, it may be deduced by combining Proposition~\ref{prop:projekce-konstrukce}, Lemma~\ref{lem:ball of L2}, Lemma~\ref{lem:estimate} and Lemma~\ref{lem:periodic-densities} in the way explained in the proof of Theorem~\ref{thm:bidual-complemented-implies-Hilbert}.

    The existence of a minimum and the converse inequality are now quite easy: Let $P:E\to V$ be a projection with $\norm{P}=\lambda(V,E)$. Define 
    $$H\g=S(P\circ\g), \quad \g\in L^1(\gamma;E).$$ 
    Then $H$ is an extension of $S$ with $\norm H\le\norm P=\lambda(V,E)$.
\end{proof}

\subsection*{Data availability statement:} No data sets were used or generated within our research.

\subsection*{Declaration of conflict of interests:} None.

\subsection*{Statement on the use of AI tools:}
The authors acknowledge the use of OpenAI's Sol 5.6 and Astra
models in developing this paper. We asked the models to suggest
solutions to intermediate problems, supplying some of our earlier
ideas and approaches. We then worked to understand and develop
the proposed arguments and posed further questions arising naturally
from this analysis. Through several iterations, this process led
to the main ideas of the proof. After working through the arguments,
we rewrote and reorganized the manuscript to reflect our understanding
and mathematical style, with particular attention to separating
the main ideas from the technical details. The authors take full
responsibility for the results and proofs presented in the paper.

\bibliographystyle{acm}
\bibliography{hilb-com}

\end{document}